\documentclass{article}

\usepackage{amssymb}
\usepackage{amsmath}
\usepackage{enumerate}
\usepackage{amscd}
\usepackage{multirow}
\usepackage{epsfig}

\newtheorem{introthm}{Theorem}
\newtheorem{introcor}[introthm]{Corollary}

\newtheorem{theorem}{Theorem}[section]

\newtheorem{corollary}[theorem]{Corollary}
\newtheorem{proposition}[theorem]{Proposition}
\newtheorem{conjecture}[theorem]{Conjecture}

\newcommand{\qw}[1]{\!\!\!#1\!\!\!}
\newcommand{\id}{\mbox{id}}

\newcommand{\Z}{\mathbb{Z}}
\newcommand{\tZ}{\tilde{\mathbb{Z}}}
\newcommand{\Q}{\mathbb{Q}}

\newcommand{\tensor}{\otimes}
\newcommand{\join}{*}

\newcommand{\sgn}{{\rm sgn}}

\newenvironment{proof}[1][Proof.]{\begin{trivlist}
\item[\hskip \labelsep {\em #1}]}{$\Box$\end{trivlist}}

\newcommand{\Tor}{{\rm Tor}}

\newcommand{\BD}[2]{\mathsf{BD}_{#1}^{#2}}
\newcommand{\M}[1][{}]{\mathsf{M}_{#1}}
\newcommand{\Symm}[1]{\mathfrak{S}_{#1}}

\newcommand{\zivaljevic}{$\check{\mathrm{Z}}$ivaljevi\'c}
\newcommand{\shareshian}{Shareshian}

\begin{document}

 \title{More Torsion in the Homology of the Matching Complex}

 \author{Jakob Jonsson\thanks{Department of Mathematics, KTH,
  SE-10044 Stockholm, Sweden}}

  \date{}

  \maketitle
  \begin{abstract}
    A matching on a set $X$ is a collection of pairwise disjoint
    subsets of $X$ of size two.
    Using computers, we analyze the integral homology of the matching
    complex $\M[n]$, which is the simplicial complex of 
    matchings on the set $\{1, \ldots, n\}$.
    The main result is the
    detection of elements of order $p$ in the homology for $p \in
    \{5,7,11,13\}$. Specifically, we show that there are elements of 
    order $5$ in the homology of $\M[n]$ for $n \ge 18$ and for $n
    \in \{14,16\}$. The only previously known value was $n = 14$, and
    in this particular case we have a new computer-free proof.
    Moreover, we show that there are elements of order $7$ in the
    homology of $\M[n]$ for all odd $n$ between $23$ and $41$ and for
    $n=30$. 
    In addition, there are elements of order $11$ in the homology of
    $\M[47]$ and elements of order $13$ in the homology of $\M[62]$. 
    Finally, we compute the ranks of the Sylow $3$- and $5$-subgroups 
    of the torsion part of $\tilde{H}_d(\M[n];\Z)$ for $13 \le n \le
    16$; a complete description of the homology already exists for $n
    \le 12$.
    To prove the results, we use a representation-theoretic approach,
    examining subcomplexes of the chain complex of $\M[n]$ obtained by
    letting certain groups act on the chain complex.
  \end{abstract}

  \section{Introduction}
  \label{intro-sec}

  Recall that a {\em matching} on a set $X$ is a collection of
  pairwise disjoint subsets of $X$ of size two.
  Using terminology from graph theory, we refer to a subset of size
  two as an {\em edge on $X$}. 
  The {\em matching complex} $\M[n]$ is the family of matchings on the
  set $\{1, \ldots, n\}$. Since $\M[n]$ is closed under deletion of
  edges, $\M[n]$ is an abstract simplicial complex. 
  
  Despite its simple definition, the topology of $\M[n]$ remains a
  mystery. Its rational simplicial homology is well-known and has
  been computed 
  by Bouc \cite{Bouc} and others \cite{Kara,RR,DongWachs}, but the
  integral homology is known only in special cases. 
  Specifically, the bottom nonvanishing homology group of $\M[n]$ is
  known to be an elementary $3$-group for almost all $n$
  \cite{Bouc,ShWa}. In fact, there is a nearly complete
  characterization of all $(n,d)$ such that the homology
  group $\tilde{H}_d(\M[n];\Z)$ contains elements of order $3$
  \cite{bettimatch}; see Proposition~\ref{kr-prop} (4) and (5)
  for a summary.  
  As for the existence of elements of order $p$ for primes different
  from $3$, nothing is known besides the recent discovery
  \cite{matchtor} that $\tilde{H}_4(\M[14];\Z)$ contains elements of
  order $5$. 

  Computers have not been able to tell us much more about the
  homology of $\M[n]$. A complete description is known only for $n \le
  12$; already for $n=13$, there are too many cells in $\M[n]$ to
  allow for a direct homology computation, at least with the existing
  software.

  The goal of this paper is to cut $\M[n]$ into smaller
  pieces and then use a computer to search for torsion in the homology
  of those pieces. Via this technique, we come fairly close to a
  complete description of the homology of $\M[n]$ for $13 \le n \le
  16$. Moreover, and maybe more importantly, for each of the primes 
  $5$, $7$, $11$, and $13$, we find new values of $n$ such that 
  the homology of $\M[n]$ contains elements of order $p$.

  More precisely, we do the following.
  \begin{itemize}
  \item
    First, we split the chain complex $\mathcal{C}(\M[n];\Z)$ of
    $\M[n]$ into smaller pieces with the property that the homology of
    $\M[n]$ is isomorphic to the direct sum of the homology of the 
    smaller complexes, except that the Sylow $2$-subgroups of the
    torsion part may differ. Using this splitting technique, we
    determine the $3$- and $5$-ranks of the homology of $\M[n]$ for
    $13 \le n \le 16$. Here, we define the $p$-rank of an abelian
    group to be the rank of the Sylow $p$-subgroup of the torsion part
    of the group. 
  \item
    Second, we use a similar technique to produce even smaller
    pieces. While it does not seem to be possible to compute the 
    entire homology of $\M[n]$ from these pieces, we may still
    deduce useful information about the existence of elements of order
    $p$ for
    various primes $p$. Specifically, our computations show that there
    are elements of order $5$ in
    $\tilde{H}_{4+u}(\M[14+2u];\Z)$ for $0 \le u \le 8$, in
    $\tilde{H}_{6+u}(\M[19+2u];\Z)$ for $0 \le u \le 4$, and 
    in $\tilde{H}_{8}(\M[24];\Z)$.
    This turns out to imply that there are elements of order $5$ in 
    the homology of $\M[n]$ whenever $n =14$, $n=16$, or $n \ge 18$.
    Moreover, there are elements of order $7$ in
    $\tilde{H}_{8+u}(\M[23+2u];\Z)$ for $0 \le u \le 9$
    and in $\tilde{H}_{11}(\M[30];\Z)$, elements of order $11$ 
    in $\tilde{H}_{13}(\M[47];\Z)$, and elements of order $13$ in
    $\tilde{H}_{19}(\M[62];\Z)$.
  \end{itemize}
  Almost all results are computer-based, but in Section~\ref{m14-sec}
  we present a computer-free proof that $\tilde{H}_4(\M[14];\Z)$
  contains elements of order $5$.

  Based on our computations, we derive conjectures about the
  existence of elements of order $p$ in the homology of $\M[n]$ for
  arbitrary odd primes $p$. For example, the above data suggest 
  that
  there are elements of order $p$ in the homology of $\M[(p^2+6p+1)/4]$
  for any odd prime $p$.

  Let us explain the underlying method in some detail.
  For a given group $G$ acting on the simplicial chain complex
  $\mathcal{C}(\M[n];\Z)$ of $\M[n]$, let
  $\tilde{C}_d(\M[n];\Z)/G$ be the subgroup of the chain group
  $\tilde{C}_d(\M[n];\Z)$ consisting of 
  all sums $\sum_{g \in G} g(c)$ such that $c \in
  \tilde{C}_d(\M[n];\Z)$.
  As long as the order of $G$ is not a multiple
  of a given prime $p$, the homology of the subcomplex does not
  contain elements of order $p$ unless the homology of $\M[n]$
  contains elements of order $p$; see Section~\ref{genconstr-sec}. In
  particular, if we indeed detect elements of order $p$ 
  in the subcomplex, then we can deduce that the homology of $\M[n]$
  also contains elements of order $p$. 

  To give an example, we first introduce some notation. For a sequence
  $\lambda = (\lambda_1, \ldots, \lambda_r)$ of nonnegative integers
  summing to $n$, let $\Symm{\lambda}$ be the Young group
  $\Symm{\lambda_1} \times \cdots \times
  \Symm{\lambda_r}$.  
  Given a set partition $(U_1, \ldots,  U_r)$ of $\{1, \ldots, n\}$
  such that $|U_a| = \lambda_a$ for $1 \le a \le r$, we obtain a
  natural action on $\mathcal{C}(\M[n];\Z)$ by letting
  $\Symm{\lambda_a}$ act on $U_a$ 
  in the natural manner for each $a$; for a group element $\pi \in
  \Symm{\lambda}$, the action is given by replacing the edge $xy
  = \{x,y\}$ with the edge $\pi(x)\pi(y)$ for each choice of distinct
  elements $x, y \in \{1, \ldots, n\}$.

  In Section~\ref{lambdascc-sec}, we demonstrate that 
  $\mathcal{C}(\M[n];\Z)/\Symm{\lambda}$ is isomorphic to the chain
  complex of a certain simplicial complex $\BD{r}{\lambda}$. The
  vertices of this complex are all edges and loops on $\{1, \ldots,
  r\}$; a loop is a multiset of the form $xx = \{x,x\}$).   
  A collection of edges and loops is a face of $\BD{r}{\lambda}$ if 
  and only if the total number of occurrences of the element $a$ in
  the edges and loops is at most $\lambda_a$ for $1 \le a
  \le r$. By a computer-based 
  result due to Andersen \cite{Andersen},
  $\tilde{H}_4(\BD{7}{2};\Z)$ is a cyclic group of order five. As a
  consequence, the following is true.
  \begin{proposition}[Jonsson \cite{matchtor}]
    \label{matchtor-prop}
    $\tilde{H}_{4}(\M[14]; \Z)$ contains elements of order $5$.
  \end{proposition}
  The proof \cite{matchtor} of Proposition~\ref{matchtor-prop}
  uses a definition of
  $\mathcal{C}(\M[n];\Z)/\Symm{\lambda}$ that differs slightly from
  the one given above. As alluded to earlier, we present a
  computer-free proof of the proposition in Section~\ref{m14-sec}.

  Using similar but refined methods, we have managed to deduce
  more information about the presence of elements of order $5$ in the
  homology of $\M[14]$ and other matching complexes. 
  \begin{introthm}
    By computer calculations, the following properties hold.
    \begin{itemize}
    \item[{\em(i)}]
      The $5$-rank of
      $\tilde{H}_4(\M[14];\Z)$ is $233$.
    \item[{\em(ii)}]
      The $5$-rank of 
      $\tilde{H}_5(\M[16];\Z)$ is $8163$.
    \item[{\em(iii)}]
      $\tilde{H}_{4+u}(\M[14+2u];\Z)$ contains 
      elements of order $5$ for $0 \le u \le 8$, 
      as do $\tilde{H}_{6+u}(\M[19+2u];\Z)$ for 
      $0 \le u \le 4$ and 
      $\tilde{H}_{8}(\M[24];\Z)$.
    \end{itemize}
    \label{5tor-thm}
  \end{introthm}
  Note that $233$ is a Fibonacci number. As we will see in
  Section~\ref{upto16-sec}, this does not appear to be a pure
  coincidence.

  To prove the theorem, we consider various actions of the Young group
  $\Symm{\lambda}$ on the chain complex of $\M[n]$ for certain 
  choices of $\lambda = (\lambda_1, \ldots, \lambda_r)$. Specifically,
  for each sequence ${\sf s} = (s_1, \ldots, s_r)$ of signs, we obtain 
  an action given by 
  \begin{equation}
    (\pi,c) \mapsto 
    \prod_{1 \le a \le r : s_a = -1} \sgn(\pi_a) \cdot \pi(c),
    \label{youngaction-eq}
  \end{equation}
  where $\pi = \pi_1 \cdots \pi_r$ and $\pi_a \in \Symm{U_a}$. Here,
  $\Symm{U_a}$ denotes the symmetric group $\Symm{\lambda_a}$, viewed
  as the group of permutations on the set $U_a$.
  Whenever we refer to an action on $\M[n]$ as being induced by the 
  pair $(\lambda,{\sf s})$, we mean this action.
  For clarity, we will often write the pair in matrix form as
  \begin{equation}
    \left(\begin{array}{c} \lambda
        \\ {\sf s} \end{array}\right) = 
    \left(\begin{array}{cccc} {\lambda_1} & {\lambda_2} &
        {\cdots} & {\lambda_r}
        \\ 
        {s_1} & {s_2} & {\cdots} & {s_r}
      \end{array}\right).
    \label{lambdasnot-eq}
  \end{equation}
  If all $s_a$ are equal to $-1$, then the action is simply
  $(\pi,c)\mapsto \sgn(\pi) \cdot \pi(c)$. We refer to this action as
  the {\em natural signed action} and the action 
  $(\pi,c)\mapsto \pi(c)$ as the {\em natural unsigned action}.

  To prove (i) and (ii) in Theorem~\ref{5tor-thm}, we split the chain
  complexes into many small 
  pieces using the actions defined in (\ref{youngaction-eq}) for the
  group $(\Symm{2})^r$ for $2r = 14$ and $2r = 16$, respectively. 
  For the proof of (iii), we consider the natural signed action of
  certain Young groups $\Symm{\lambda}$ on the
  chain complex of $\M[n]$. In some
  cases, we need to pick an even larger group whose elements 
  permute not only the elements within each $U_a$ but also the sets
  $U_a$ themselves. We describe this construction in greater detail
  in Section~\ref{wreath-sec}.

  It is important to stress that the choice of signs has a significant
  impact on the homology of the resulting piece. For many of the
  choices of $\lambda$ used for proving Theorem~\ref{5tor-thm}, the
  piece resulting from the natural {\em signed} action of
  $\Symm{\lambda}$ contains plenty of elements of order $5$, whereas
  the piece resulting from the {\em unsigned} action of
  $\Symm{\lambda}$ does not contain any such elements at all.

  By the following proposition, the result that
  $\tilde{H}_6(\M[19];\Z)$ and $\tilde{H}_8(\M[24];\Z)$ contain
  elements of order $5$ turns out to be of particular importance.
  \begin{proposition}[Jonsson \cite{bettimatch}]
    For $q \ge 3$,
    if $\tilde{H}_{2q}(\M[5q+4];\Z)$ contains
    elements of order $5$, then so does
    $\tilde{H}_{2q+u}(\M[5q+4+2u];\Z)$ for each
    $u \ge 0$. 
    \label{further5-prop}
  \end{proposition}
  Theorem~\ref{5tor-thm} and Proposition~\ref{further5-prop} imply the
  following.
  \begin{introthm}
    The homology groups 
    $\tilde{H}_{6+u}(\M[19+2u];\Z)$ 
    and 
    $\tilde{H}_{8+u}(\M[24+2u];\Z)$ 
    contain  
    elements of order $5$ for each
    $u \ge 0$. Thus the homology of $\M[n]$ contains
    elements of order $5$ for all 
    $n \ge 18$ and for $n=14$ and $n=16$.
    \label{m19+2u-thm}
  \end{introthm}
  We do not know whether
  Proposition~\ref{further5-prop} can be extended to include $q=2$.
  In particular, we cannot tell whether
  $\tilde{H}_{4+u}(\M[14+2u];\Z)$
  contains elements of order $5$ for $u \ge 9$, though we do
  conjecture that this is the case.

  Next, we consider higher torsion, in which case we have the
  following results.
  \begin{introthm}
    By computer calculations, the following properties hold.
    \begin{itemize}
    \item[{\em(i)}]
      The homology group
      $\tilde{H}_{8+u}(\M[23+2u];\Z)$ contains
      elements of order $7$ for $0 \le u \le 9$.
      The same is true for the group
      $\tilde{H}_{11}(\M[30];\Z)$.
    \item[{\em(ii)}]
      The homology group
      $\tilde{H}_{19}(\M[47];\Z)$ contains
      elements of order $11$.
    \item[{\em(iii)}]
      The homology group
      $\tilde{H}_{26}(\M[62];\Z)$ contains
      elements of order $13$.
    \end{itemize}
    \label{7tor-thm}
  \end{introthm}
  
  To prove Theorem~\ref{7tor-thm} (i) for $u=0$, we consider the action
  induced by 
  $
\left(\begin{array}{c} \lambda
      \\[-1ex] {\sf s} \end{array}\right) = 
\left(\begin{array}{cccccc} \qw{3} & \qw{4} & \qw{4} & \qw{4} &
      \qw{4} & \qw{4} 
      \\[-1ex] \qw{+} & \qw{-} & \qw{-} & \qw{-} & \qw{-} & \qw{-}
    \end{array}\right)$ 
  as defined in (\ref{youngaction-eq}).
  A computer calculation yields that
  the homology in degree eight of the resulting
  chain complex is a group of order $7$.  

  The action induced by $\left(\begin{array}{cccccc} \qw{3} & \qw{4} &
      \qw{4} & \qw{4} & 
      \qw{4} & \qw{4} 
      \\[-1ex] \qw{+} & \qw{-} & \qw{-} & \qw{-} & \qw{-} & \qw{-} \end{array}\right)$ 
  fits nicely into a pattern starting 
  with 
  $\left(\begin{array}{cccc} \qw{1} & \qw{2} & \qw{2} & \qw{2} \\[-1ex]
      \qw{+} & \qw{-} & \qw{-} & \qw{-} \end{array}\right)$ 
  and 
  $\left(\begin{array}{ccccc} \qw{2} & \qw{3} & \qw{3} & \qw{3} &
      \qw{3}
      \\[-1ex] \qw{+} & \qw{-} & \qw{-} & \qw{-} & \qw{-}
    \end{array}\right)$.
  The action induced by 
  $\left(\begin{array}{cccc} \qw{1} & \qw{2} & \qw{2} & \qw{2} \\[-1ex]
      \qw{+} & \qw{-} & \qw{-} & \qw{-} \end{array}\right)$ on $\M[7]$
  yields a chain complex with a homology group of order $3$ in degree
  one, whereas the corresponding 
  action induced by 
  $\left(\begin{array}{ccccc} \qw{2} & \qw{3} & \qw{3} & \qw{3} &
      \qw{3}
      \\[-1ex] \qw{+} & \qw{-} & \qw{-} & \qw{-} & \qw{-}
    \end{array}\right)$
  on $\M[14]$ induces a chain complex with a homology group of order
  $5$ in degree four.

  Parts (ii) and (iii) in Theorem~\ref{7tor-thm} fit the same pattern;
  we pick the group actions induced by 
  $\left(\begin{array}{cccccccc} \qw{5} & \qw{6} & \qw{6} & \qw{6} 
      & \qw{6} & \qw{6} & \qw{6} & \qw{6}
      \\[-1ex] \qw{+} & \qw{-} & \qw{-} & \qw{-}
      & \qw{-} & \qw{-} & \qw{-} & \qw{-}
    \end{array}\right)$
  and 
  $\left(\begin{array}{ccccccccc} \qw{6} & \qw{7} & \qw{7} & \qw{7} 
      & \qw{7} & \qw{7} & \qw{7} & \qw{7} & \qw{7}
      \\[-1ex] \qw{+} & \qw{-} & \qw{-} & \qw{-}
      & \qw{-} & \qw{-} & \qw{-} & \qw{-} & \qw{-}
    \end{array}\right)$. The resulting chain complexes are way too
  large for a computer to handle, but since almost all $\lambda_a$ are
  the same, we may extend the group to include elements that
  permute the sets $U_a$; see the discussion above
  after Theorem~\ref{5tor-thm}. 

  For completeness, let us mention that the corresponding group action
  induced by 
  $\left(\begin{array}{ccccccc} \qw{4} & \qw{5} & \qw{5} & \qw{5} 
      & \qw{5} & \qw{5} & \qw{5}
      \\[-1ex] \qw{+} & \qw{-} & \qw{-} & \qw{-}
      & \qw{-} & \qw{-} & \qw{-}
    \end{array}\right)$ yields a cyclic homology group of order $9$ in 
  degree $13$. Since the order of the acting group 
  is a multiple of $9$, we cannot conclude anything about the
  existence of elements of order $9$ in the homology of $\M[34]$.

  \begin{table}[htb]
    \caption{The homology $\tilde{H}_d(\M[n];\Z)$ for $n \le 16$. For
        $n \ge 13$, 
      the second highest nonvanishing homology
      group is a guess based on computations over some small fields
      $\Z_p$ of odd characteristic; the torsion part might be a
      strictly larger group.} 
    \begin{footnotesize}
    \begin{center}
        \begin{tabular}{|r||c|c|c|c|c|}
          \hline
          & & & & & \\[-2.5ex]
          &  $d=0$ & \ 1 \ & \ 2 \ & \ 3 \ & \ 4 \
          \\
          & & & & & \\[-2.5ex]
          \hline
          \hline
          & & & & & \\[-2.5ex]
          $n = 3$ & $\Z^2$ & $-$ & $-$ & $-$ & $-$ \\
          & & & & & \\[-2.5ex]
          \hline
          & & & & & \\[-2.5ex]
          $4$ & $\Z^2$ & $-$ & $-$ & $-$ & $-$ \\
          & & & & & \\[-2.5ex]
          \hline
          & & & & & \\[-2.5ex]
          $5$ & $-$ & $\Z^6$ & $-$ & $-$ & $-$ \\
          & & & & & \\[-2.5ex]
          \hline
          & & & & & \\[-2.5ex]
          $6$ & $-$ & $\Z^{16}$ & $-$ & $-$ & $-$ \\
          & & & & & \\[-2.5ex]
          \hline
          & & & & & \\[-2.5ex]
          $7$ & $-$ & $\Z_3$ & $\Z^{20}$ & $-$ & $-$\\
          & & & & & \\[-2.5ex]
          \hline
          & & & & & \\[-2.5ex]
          $8$  & $-$ & $-$ & $\Z^{132}$ & $-$ & $-$ \\
          & & & & & \\[-2.5ex]
          \hline
          & & & & & \\[-2.5ex]
          $9$  & $-$ & $-$ & $\Z^{42} \oplus (\Z_3)^8$ &
          $\Z^{70}$ & $-$ \\
          & & & & & \\[-2.5ex]
          \hline
          & & & & & \\[-2.5ex]
          $10$  & $-$ & $-$ & $\Z_3$ &
          $\Z^{1216}$ & $-$  \\
          & & & & & \\[-2.5ex]
          \hline
          & & & & & \\[-2.5ex]
          $11$ & $-$ & $-$ & $-$ & $\Z^{1188} \oplus (\Z_3)^{45}$ &
          $\Z^{252}$ \\
          & & & & & \\[-2.5ex]
          \hline
          & & & & & \\[-2.5ex]
          $12$ & $-$ & $-$ & $-$ & $(\Z_3)^{56}$  &
          $\Z^{12440}$ \\[-2.5ex]
          & & & & & \\
          \hline
        \end{tabular}

        \bigskip

        \begin{tabular}{|r||c|c|c|c|}
          \hline
          & \ $d=3$ \ & \ 4 \ & \ 5 \ & \ 6 \ \\
          & & & & \\[-2.5ex]
          \hline
          \hline
          & & & & \\[-2.5ex]
          $n=13$ &  $\Z_3$  &
          $\Z^{24596} \oplus (\Z_3)^{220}$ &
          $\Z^{924}$ & $-$ \\  
          & & & & \\[-2.5ex]
          \hline
          & & & & \\[-2.5ex]
          $14$  & $-$ &
          $ (\Z_5)^{233} \oplus
          (\Z_3)^{2157}$ &  $\Z^{138048}$ & $-$ \\ 
          & & & & \\[-2.5ex]
          \hline
          & & & & \\[-2.5ex]
          $15$
          & $-$ & $(\Z_3)^{92}$ & $\Z^{472888} \oplus (\Z_3)^{1001}$
          & 
          $\Z^{3432}$ \\ 
          & & & & \\[-2.5ex]
          \hline
          & & & & \\[-2.5ex]
          $16$
          & $-$ & $\Z_3$ & 
          $\begin{array}{c}\Z^{24024}\ \oplus \\ (\Z_5)^{8163} \oplus
          (\Z_3)^{60851}\end{array}$ &
          $\Z^{1625288}$ \\[-2.5ex]
          & & & & \\
          \hline
        \end{tabular}
    \end{center}
    \end{footnotesize}
    \label{matching-tab}
  \end{table}

  In addition, we establish precise results about 
  the $3$-rank of the homology of $\M[n]$ for $13 \le n \le 16$. The
  methods used are 
  exactly the same as those 
  used for proving Theorem~\ref{5tor-thm} (i) and (ii).
  \begin{introthm}
    By computer calculations, the following properties hold.
    \begin{itemize}
    \item[\em(i)]
      The $3$-rank of $\tilde{H}_4(\M[13];\Z)$ is $220$.
    \item[\em(ii)]
      The $3$-rank of $\tilde{H}_4(\M[14];\Z)$ is $2157$.
    \item[\em(iii)]
      The $3$-ranks of $\tilde{H}_4(\M[15];\Z)$ and 
      $\tilde{H}_5(\M[15];\Z)$ are $92$ and $1001$, respectively.
    \item[\em(iv)]
      The $3$-rank of $\tilde{H}_5(\M[16];\Z)$ is $60851$.
    \end{itemize}
    \label{3tor-thm}
  \end{introthm}
  It is well-known \cite{ShWa} that 
  $\tilde{H}_4(\M[15];\Z)$ is an elementary $3$-group. We do not know
  whether the other $3$-groups are elementary.
  See Table~\ref{matching-tab} for a summary of the situation for $n
  \le 16$.

  \begin{table}[htb]
    \caption{List of known infinite and prime orders
      of elements in the group $\tilde{H}_{d}(\M[n];\Z)$ for 
      $17 \le n \le 28$. Legend: $\infty$ = 
      group is infinite; $p$ = group contains
      elements of order $p$; $(p)$ = group is conjectured to contain
      elements of order $p$.} 
    \begin{footnotesize}
    \begin{center}
        \begin{tabular}{|r||c|c|c|c|c|c|c|c|}
          \hline
          & $d=5$ & \ 6 \ & \ 7 \ & \ 8 \ & \ 9 \ & \ 10 \ & \ 11 \ &
          12 \\ 
          & & & & & & & & \\[-2.5ex]
          \hline
          \hline
          & & & & & & & & \\[-2.5ex]
          $n=17$ &  $3$  &
          $\infty, (3)$ &
          $\infty$ & $-$ & $-$ & $-$ & $-$ & $-$ 
          \\  
          & & & & & & & & \\[-2.5ex]
          \hline
          & & & & & & & & \\[-2.5ex]
          $18$  & $3$ &
          $\infty, 3, 5$ & $\infty$ & $-$ & $-$ & $-$ & $-$  & $-$ 
          \\  
          & & & & & & & & \\[-2.5ex]
          \hline
          & & & & & & & & \\[-2.5ex]
          $19$  & $3$ &
          $3, 5$ & $\infty, (3)$ & $\infty$ & $-$  & $-$  & $-$  & $-$ 
          \\  
          & & & & & & & & \\[-2.5ex]
          \hline
          & & & & & & & & \\[-2.5ex]
          $20$  & $-$ & $3$ &
          $\infty, 3, 5$ & $\infty$ & $-$  & $-$  & $-$  & $-$ 
          \\  
          & & & & & & & & \\[-2.5ex]
          \hline
          & & & & & & & & \\[-2.5ex]
          $21$  & $-$ & $3$ &
          $3, 5$ & $\infty, (3)$ & $\infty$ & $-$  & $-$  & $-$ 
          \\  
          & & & & & & & & \\[-2.5ex]
          \hline
          & & & & & & & & \\[-2.5ex]
          $22$  & $-$ & $3$ &
          $3$ & $\infty, 3,5$ & $\infty$ & $-$  & $-$  & $-$ 
          \\  
          & & & & & & & & \\[-2.5ex]
          \hline
          & & & & & & & & \\[-2.5ex]
          $23$  & $-$ & $-$ & $3$ &
          $3,5,7$ & $\infty, (3)$ & $\infty$ & $-$  & $-$ 
          \\  
          & & & & & & & & \\[-2.5ex]
          \hline
          & & & & & & & & \\[-2.5ex]
          $24$  & $-$  & $-$ & $3$ &
          $3,5$ & $\infty, 3, 5$ & $\infty$ & $-$  & $-$ 
          \\  
          & & & & & & & & \\[-2.5ex]
          \hline
          & & & & & & & & \\[-2.5ex]
          $25$  & $-$ & $-$ & $3$ &
          $3$ & $\infty, 3,5,7$ & $\infty, (3)$ & $\infty$ & $-$ 
          \\  
          & & & & & & & & \\[-2.5ex]
          \hline
          & & & & & & & & \\[-2.5ex]
          $26$  & $-$ & $-$ & $-$ & $3$ &
          $3,5$ & $\infty,3,5$ & $\infty$ & $-$ 
          \\  
          & & & & & & & & \\[-2.5ex]
          \hline
          & & & & & & & & \\[-2.5ex]
          $27$  & $-$ & $-$ & $-$ & $3$ &
          $3$ & $\infty,3,5,7$ & $\infty, (3)$ & $\infty$
          \\  
          & & & & & & & & \\[-2.5ex]
          \hline
          & & & & & & & & \\[-2.5ex]
          $28$  & $-$ & $-$ & $-$ & $3$ &
          $3$ & $3,5$ & $\infty, 3, 5$ & $\infty$
          \\[-2.5ex]
          & & & & & & & & \\
          \hline
        \end{tabular}
    \end{center}
    \end{footnotesize}
    \label{matchingnd-fig}
  \end{table}

  Since before, it is known that $\tilde{H}_1(\M[7];\Z) \cong \Z_3$, 
  $\tilde{H}_2(\M[9];\Z) \cong \Z^{42} \oplus (\Z_3)^8$, 
  and 
  $\tilde{H}_3(\M[11];\Z) \cong \Z^{1188} \oplus (\Z_3)^{45}$. 
  This suggests the following conjecture.
  \begin{conjecture}
    For $k \ge 0$, we have that 
    $\tilde{H}_{k+1}(\M[2k+7];\Z)$ is the direct sum of 
    a free group and an elementary $3$-group 
    of rank $\binom{2k+6}{k}$.
    \label{3torrank-conj}
  \end{conjecture}
  See Table~\ref{matchingnd-fig} for a schematic overview of the
  situation for $17 \le n \le 28$.

  The proofs of
  Theorems~\ref{5tor-thm}, \ref{7tor-thm}, and \ref{3tor-thm} involve
  chain complexes that are not simplicial.
  To handle such chain complexes, we use Pilarczyk's
  excellent computer program {\sf HOMCHAIN} (version 2.08), which is
  part of the advanced version of the {\sf CHomP} package
  \cite{Pilar}. 

  \subsection{The big picture}
  \label{big-sec}

  Note that the first occurrence
  of elements of order $3$ in $H_d(\M[n];\Z)$ is for $(n,d) = (7,1)$ and
  the first occurrence of elements of order $5$ is for $(n,d) = (14,4)$.
  Let us provide a heuristic argument explaining why it seemed
  reasonable to look for elements of order $7$, $11$,
  and $13$ in $H_d(\M[n];\Z)$ for $(n,d) = (23,8)$, $(47,19)$,
  and $(62,26)$, respectively. The argument is best understood when 
  expressed in terms of a pair $(k,r)$ of parameters, introduced in a
  previous paper \cite{bettimatch}, defined as
  \begin{equation}
    \left\{
    \begin{array}{l}
      k = 3d-n+4
      \\
      r = n-2d-3 
    \end{array}
    \right. \Longleftrightarrow
    \left\{
    \begin{array}{l}
    n = 2k+1+3r \\
    d = k-1+r.
    \end{array}
    \right.   
    \label{nd2kr-eq}
  \end{equation}
  The following proposition and corollary provide rationale for this
  parameter choice. 
  \begin{proposition}
    Let $n \ge 1$. Then the following hold for the homology group
    $H_d(\M[n];\Z) = \tilde{H}_{k-1+r}(\M[2k+1+3r];\Z)$.
    \begin{itemize}
    \item[{\rm (1)}]
      {\rm(Bouc {\rm\cite{Bouc}})} 
      The group is infinite if and only if $k \ge
      r(r-1)/2$ and $r \ge 0$.  
    \item[{\rm (2)}]
      {\rm(Bj\"orner et  al$.$ \cite{BLVZ})}
      The group is zero unless $k \ge 0$ and $r \ge 0$. 
    \item[{\rm (3)}]
      {\rm(Bouc \cite{Bouc}; Shareshian \& Wachs \cite{ShWa})}
      The group is a nonvanishing elementary $3$-group whenever 
      $k \in \{0,1,2\}$ and $r \ge k+2$.
    \item[{\rm (4)}]
      {\rm(Jonsson \cite{bettimatch})}
      The group is a nonvanishing $3$-group whenever $0 \le k \le
      r-2$. 
    \item[{\rm (5)}]
      {\rm(Jonsson \cite{bettimatch})}
      The group contains elements of order $3$
      whenever $k \ge 0$ and $r \ge 3$.
    \end{itemize}
    \label{kr-prop}
    \end{proposition}

  \begin{corollary}
    For $n \ge 1$,
    the group 
    $\tilde{H}_d(\M[n];\Z) = \tilde{H}_{k-1+r}(\M[2k+1+3r];\Z)$
    is nonzero if and only if $k \ge 0$ and $r \ge 0$. 
    \label{kr-cor}
  \end{corollary}
  Regarding Proposition~\ref{kr-prop} (4)-(5), we conjecture that 
  the following stronger properties hold.
  \begin{conjecture}
    $\tilde{H}_d(\M[n];\Z) = \tilde{H}_{k-1+r}(\M[2k+1+3r];\Z)$
    contains elements of order $3$ if and only if $k
    \ge 0$ and $r \ge 2$. The homology group is an elementary
    nonvanishing $3$-group if and only if $0 \le k \le r-2$.
    \label{3tor-conj}
  \end{conjecture}
  In Conjecture~\ref{3torrank-conj} above, we conjectured
  that the $3$-rank of $\tilde{H}_{k+1}(\M[2k+7];\Z)$ is
  $\binom{2k+6}{k}$ for $k \ge 0$. This would
  imply the first part of Conjecture~\ref{3tor-conj}; it is well-known
  and easy to prove that 
  $\tilde{H}_{k-1+r}(\M[2k+1+3r];\Z)$ is free if $r \le 1$
  (equivalently, $d \ge \frac{n-4}{2}$). Conjecture~\ref{5tor-conj}
  below implies the ``only if'' part of the second statement.

  \begin{table}[ht]
    \begin{center}
        \begin{tabular}{|r||c|c|c|c|c|c|c|c|c|c|c|}
          \hline
          &  $k=0$ & \ 1 \ & \ 2 \ & \ 3 \ & \ 4 \ & \ 5 \ & \ 6 \ & \
          7 \ & \ 8 \ & \ 9 \ & \ 10 \
          \\
          \hline
          \hline
          $r = 0$ & $\infty$ & $\infty$ & $\infty$ & $\infty$ &
          $\infty$ & $\infty$ & $\infty$ & $\infty$ & $\infty$ &
          $\infty$ & $\infty$ \\ 
          \hline
          $1$ & $\infty$ & $\infty$ & $\infty$ & $\infty$ 
          & $\infty$ &  $\infty$ & $\infty$ &  $\infty$ & $\infty$ &
          $\infty$  &  $\infty$ \\  
          \hline
          $2$ &  $3$  & $\infty$ & $\infty$ & $\infty$ &
          $\infty$ & $\infty$ & $\infty$ & $\infty$ & $\infty$ &
          $\infty$  &  $\infty$ \\  
          \hline
            $3$  &     &      &  $5$  & $\infty$ & $\infty$
          & $\infty$ & $\infty$ & $\infty$ & $\infty$ & $\infty$ &
          $\infty$ \\  
          \hline
          $4$  &  & &  &
           & & $7$ & $\infty$  & $\infty$ & $\infty$  & $\infty$ &
          $\infty$ \\
          \hline
          $5$  &  & &  &
           & &  &   & &   & $(*)$ &
          $\infty$ \\
          \hline
        \end{tabular}
    \end{center}
    \caption{Boxes marked with ``$\infty$'' correspond to pairs $(k,r)$
      such that $\tilde{H}_{k-1+r}(\M[2k+1+3r];\Z)$ is infinite.
      We have marked boxes corresponding to the first known
      occurrences of elements of order $p$ for $p = 3,5,7$. The next box 
      to form a corner of the region of finite homology is marked with
      ``$(*)$''.}  
   \label{matchingkr1-fig}
  \end{table}

  Note that $(n,d) = (7,1)$ corresponds to $(k,r) = (0,2)$ and that 
  $(n,d) = (14,4)$ corresponds to $(k,r) = (2,3)$. Both these pairs
  satisfy the equation 
  \begin{equation}
    k = r(r-1)/2-1.
    \label{barelybouc-eq}
  \end{equation}
  In particular, the pairs just barely fail to meet the requirement in  
  Proposition~\ref{kr-prop} (1). An alternative way to put it is
  to say that the pairs $(k,r)$ satisfying
  (\ref{barelybouc-eq}) are the corners of the region 
  of finite homology as illustrated in Table~\ref{matchingkr1-fig}.
  The next corner after $(0,2)$ and $(2,3)$ is $(5,4)$, which
  yields $(n,d) = (23,8)$, exactly the pair for which we detected
  elements of order $7$. The subsequent corners are $(9,5)$, $(14,6)$, and
  $(20,7)$, which yield $(n,d) = (34,13)$, $(47,19)$, and $(62,26)$,
  respectively. The latter two pairs are the places where we detected
  elements of order $11$ and $13$, respectively. 

  One may ask whether there are elements of order $2r-1$ in the group
  corresponding to the pair $(k,r) = (r(r-1)/2-1,r)$ for all
  $r \ge 2$ or at least all $r$ such that $2r-1$ is a prime.
  As already alluded to, we do not know whether there are elements of
  order $9$ in the group corresponding to $(k,r) = (9,5)$.

  Expressing Theorems~\ref{5tor-thm} and \ref{m19+2u-thm} in terms of
  the parameters $k$ and $r$, we obtain the following characterization
  of groups known to contain elements of order $5$.
  \begin{introcor}
    The homology group $\tilde{H}_d(\M[n];\Z) =
    \tilde{H}_{k-1+r}(\M[2k+1+3r];\Z)$ 
    contains elements of order $5$ whenever either of the following
    holds.
    \begin{itemize}
    \item 
      $r=3$ and $2 = r-1 \le k \le 10$.
    \item
      $r \in \{4,5\}$ and $k \ge r-1$.
    \end{itemize}
  \end{introcor}
  This suggests the following conjecture.
  \begin{conjecture}
    $\tilde{H}_{k-1+r}(\M[2k+1+3r];\Z)$
    contains elements of order $5$ whenever 
    $r \ge 3$ and $k \ge r-1$. Equivalently, $\tilde{H}_d(\M[n];\Z)$
    contains elements of order $5$ whenever
    \[
    \frac{2n-8}{5} \le d \le \frac{n-6}{2}.
    \]
    \label{5tor-conj}
  \end{conjecture}
  Proceeding to the next prime, Theorem~\ref{7tor-thm} yields that
  there are elements of order $7$ in $\tilde{H}_{k-1+r}(\M[2k+1+3r];\Z)$
  for $r = 4$ and $5 \le k \le 14$ and also for $(k,r) = (7,5)$. While
  this is 
  very little evidence for a conjecture, we do hope that the
  following is true.
  \begin{conjecture}
    $\tilde{H}_{k-1+r}(\M[2k+1+3r];\Z)$
    contains elements of order $7$ whenever 
    $r \ge 4$ and $k \ge 2r-3$. Equivalently, $\tilde{H}_d(\M[n];\Z)$ 
    contains elements of order $7$ whenever
    \[
    \frac{3n-13}{7} \le d \le \frac{n-7}{2}.
    \]
  \end{conjecture}
  
  \begin{table}[htb]
    \caption{List of known infinite and prime orders
      of elements in the group $\tilde{H}_{k-1+r}(\M[2k+1+3r];\Z)$ for 
      $k \le 8$ and $r \le 7$; notation is as in 
      Table~\ref{matchingnd-fig}.} 
    \begin{footnotesize}
    \begin{center}
        \begin{tabular}{|r||c|c|c|c|c|c|c|c|c|}
          \hline
          & $k=0$ &  1  &  2  &  3  &  4  &  5  &  6  &  7 & 8
          \\
          \hline
          \hline
          $r=0$ & $\infty$ & $\infty$ & $\infty$ & $\infty$ &
          $\infty$ & $\infty$ & $\infty$ & $\infty$& $\infty$ \\
          \hline
          $1$ & $\infty$ & $\infty$ & $\infty$ & $\infty$ 
          & $\infty$ &  $\infty$ & $\infty$ &  $\infty$ & $\infty$\\ 
          \hline
          $2$ &  3  & $\infty,3 $ & $\infty,3$ & $\infty,3$ &
          $\infty,3$ & $\infty, (3)$ & $\infty, (3)$ & $\infty, (3)$
          & $\infty, (3)$\\  
          \hline
          $3$ & $3$ & $3$  & $3$ & $\infty, 3$ & $\infty, 3$
          & $\infty, 3$ & $\infty, 3$ & $\infty, 3$ & $\infty, 3$ \\ 
            &  &  & $5$ & $5$ & $5$ & $5$
          & $5$ & $5$ & $5$ \\ 
          \hline
          $4$  & $3$ & $3$ & $3$ &
          $3$ & $3$ & $3$ & $\infty,3$  & $\infty,3$ & $\infty,3$ \\
               &  &  & & $5$ & $5$ & $5, 7$  & $5, 7$  &
          $5, 7$  &
          $5, 7$ \\ 
          \hline
          $5$  & $3$ & $3$ & $3$ & $3$ & $3$ & $3$ & $3$ & $3$ & $3$ \\
               &     &     &     &       & $5$ & $5$ & $5$ &
          $5, 7$& $5, (7)$\\
          \hline
          $6$  & $3$ & $3$ & $3$ & $3$ & $3$ & $3$ & $3$ & $3$& $3$ \\
               &     &     &     &     &     & $(5)$ & $(5)$ & $(5)$ & $(5)$ \\
          \hline
          $7$  & $3$ & $3$ & $3$ & $3$ & $3$ & $3$ & $3$ & $3$ & $3$ \\
               &     &     &     &     &     &     & $(5)$ & $(5)$ & $(5)$ \\
          \hline
        \end{tabular}
    \end{center}
    \end{footnotesize}
   \label{matchingkr2-fig}
  \end{table}
  Table~\ref{matchingkr2-fig} gives a schematic overview similar to
  the one in Table~\ref{matchingnd-fig} but with rows and columns
  indexed by $r$ and $k$ rather than $n$ and $d$.

  We have even less evidence for the following conjecture about
  the existence of elements of order $p$ for general $p$ in the
  homology of $\M[n]$.
  \begin{conjecture}
    Let $p = 2q-1$ be an odd prime. Then
    $\tilde{H}_{k-1+r}(\M[2k+1+3r];\Z)$ 
    contains elements of order $p$ whenever 
    $r \ge q$ and $k \ge (q-2)r-\binom{q-1}{2}$. Equivalently,
    $\tilde{H}_d(\M[n];\Z)$ contains elements of order $p$ whenever
    \[
    \frac{(q-1)(n-q/2)}{2q-1} - 1 \le d \le \frac{n-q-3}{2}.
    \]
    \label{grand-conj}
  \end{conjecture}

  \section{General construction}
  \label{genconstr-sec}

  Let 
  \[
  \begin{CD}
    \mathcal{C} : \cdots @>\partial>> C_{d+1} @>\partial>> C_d
    @>\partial>> C_{d-1} @>\partial>> \cdots 
  \end{CD}
  \]
  be a chain complex of abelian groups.
  Let $G$ be a finite group acting on $\mathcal{C}$; thus we have 
  a map $\alpha : G \times \bigoplus_d C_d \rightarrow \bigoplus_d
  C_d$ satisfying 
  \begin{eqnarray*}
    \alpha_g(\partial(c)) &=& \partial(\alpha_g(c));\\
    \alpha_g \circ \alpha_h(c) &=& \alpha_{gh}(c)
  \end{eqnarray*}
  for all $g,h \in G$ and $c \in \bigoplus_d C_d$.
  
  Let $C_d/(G,\alpha)$ be the subgroup of $C_d$ consisting of
  all 
  elements \[
  {}[c] = \sum_{g \in G} \alpha_g(c);
  \]
  $c \in C_d$. Writing $C_d/G = C_d/(G,\alpha)$, this yields a chain
  subcomplex
  \[
  \begin{CD}
    \mathcal{C}/G : \cdots @>\partial>> C_{d+1}/G
    @>\partial>> C_d/G 
    @>\partial>> C_{d-1}/G @>\partial>> \cdots 
  \end{CD}
  \]
  \begin{proposition}
    Let $q$ be a positive integer. If $H_d(\mathcal{C}/G)$ contains
    elements of order $q$, then $H_d(\mathcal{C})$ contains elements
    of order $q/\gcd(q,|G|)$.
    In particular, if $q$ and $|G|$ are
    coprime, then $H_d(\mathcal{C})$ contains elements of order $q$.
    \label{qpreserved-prop}
  \end{proposition}
  \begin{proof}
    Defining $\iota : C_d/G \mapsto C_d$ to be the natural inclusion
    map, we obtain the following diagram of maps between homology
    groups.
    \[
    \begin{CD}
      H_d(\mathcal{C}/G) @>\iota^*>> H_d(\mathcal{C}) @>[\cdot]^*>>
      H_d(\mathcal{C}/G).
    \end{CD}
    \]
    Now, suppose that $z$ is an element of order $q$ in 
    $H_d(\mathcal{C}/G)$. Let $w$ be an element from $C_d/G$ in the
    homology class $z$. By construction, $w = [w']$ for some $w'$. 
    As a consequence, 
    \[
    [\iota(w)] = \sum_{g \in G} \sum_{h \in G} \alpha_g(\alpha_h(w'))
    = \sum_{g \in G} \sum_{h \in G} \alpha_{gh}(w')
    = \sum_{g \in G} w
    = |G| \cdot w.
    \]
    This implies that $[\iota^*(z)]^* = |G|\cdot z$, which is an
    element of order $q/\gcd(q,|G|)$. It follows that the order of
    $\iota^*(z)$ in $H_d(\mathcal{C})$ is divisible by
    $q/\gcd(q,|G|)$.
  \end{proof}

  Note that $\mathcal{C}/G$ is a subcomplex of the complex of
  $G$-invariant elements. In our applications with Young groups acting
  on matching complexes, the two complexes are not
  identical in general. The reason for examining $\mathcal{C}/G$ rather
  than the $G$-invariant complex is that the former complex turns out
  to be more attractive in our situation. For example, the
  results presented in Sections~\ref{lambdascc-sec} and
  \ref{twosteplie-sec} are not valid in general for the $G$-invariant
  complex unless $|G|$ is a unit in the underlying coefficient ring. 

  \section{A chain complex structure on a family of multisets}
  \label{genchain-sec}

  We want to understand the chain complex obtained
  from that of $\M[n]$ by acting on $\M[n]$ as in 
  (\ref{youngaction-eq}). In this section, we look at a more general
  situation, thereby postponing the special case of importance to us
  until the next section.

  A {\em loop} on a set $X$ is a multiset of the form $xx = \{x,x\}$,
  where $x \in X$.
  Let $V$ be a finite totally ordered set, and let 
  $E$ be a subset of the set of edges and loops on $X$.
  We say that two
  elements $a$ and $b$ in $V$ {\em commute} if $ab$ belongs to $E$;
  $a$ commutes with itself if the loop $aa$ belongs to $E$. Otherwise,
  $a$ and $b$ {\em anticommute}.  
  Given a sequence $(v_0, \ldots, v_d)$ of
  elements from $V$, say that the pair $(i,j)$ forms an {\em
    inversion} 
  if $i<j$ and $v_i > v_j$. The inversion is {\em commuting}
  if $v_i$ and $v_j$ commute and anticommuting otherwise.

  In what follows, we will consider multisets. Given a multiset
  $\sigma$ and an element $x \in \sigma$, we define $\sigma \setminus
  \{x\}$ to be the multiset obtained by decreasing the multiplicity of
  $x$ in $\sigma$ by one. We extend this to larger submultisets in the
  obvious manner; $\sigma \setminus \{x,y\} = (\sigma \setminus \{x\})
  \setminus \{y\}$ and so on.

  Let $\Delta$ be a family of multisets of elements from $V$
  satisfying the following properties.
  \begin{itemize}
  \item[(A)]
    If $\sigma, \tau \in \Delta$
    and $\sigma \subseteq \rho \subseteq \tau$, then
    $\rho \in \Delta$. 
  \item[(B)]
    For each multiset $\sigma$ in $\Delta$, if $x$ and $y$ commute and
    $\{x,y\} \subseteq \sigma$, then the multiset $\sigma \setminus
    \{x,y\}$ does not belong to $\Delta$. 
  \item[(C)]
    For each $\sigma \in \Delta$, any element $a \in V$ appearing more
    than once in $\sigma$ commutes with itself.
  \end{itemize}

  Our goal is to define a chain complex associated to $\Delta$ and
  $E$. 
  As we will see, if $E$ is empty, meaning that no pairs of
  elements commute, and $\Delta$ is closed under deletion of elements, 
  then the resulting chain complex coincides with the simplicial chain
  complex on $\Delta$.

  For a coefficient ring $R$, define a chain complex
  $\mathcal{C}((\Delta,E);R)$ in the following manner.
  The chain group $C_d((\Delta,E);R)$ is the free $R$-module with
  one generator 
  $x_0 \tensor \cdots \tensor x_d$ for each multiset $\sigma = \{x_0,
  \ldots, x_d\}$ in $\Delta$ of size $d+1$ ($x_0 \le \cdots \le
  x_d$).
  By convention, we set $x_0 \tensor \cdots \tensor x_d$ equal to zero
  whenever $\{x_0, \ldots, x_d\} \notin \Delta$.

  Sometimes we will need to consider generators in which the
  elements are not arranged according to the total order on $V$.
  For a permutation $\pi \in \Symm{\{0, \ldots, d\}}$, define
  \begin{equation}
    x_{\pi(0)} \tensor \cdots \tensor x_{\pi(d)} =
    (-1)^{\eta} \cdot x_0 \tensor \cdots \tensor x_d, 
    \label{permx-eq}
  \end{equation}
  where $\eta$ is the number of anticommuting inversions 
  of $(x_{\pi(0)}, \ldots, x_{\pi(d)})$. 
  Equivalently, for any sequence $(x_0, \ldots, x_d)$, ordered or not, 
  and any integer $i$, we have that the element
  $\gamma' = x_0 \tensor \cdots \tensor x_{i-2} \tensor
  x_{i} \tensor x_{i-1} \tensor x_{i+1} \tensor \cdots \tensor x_d$
  obtained by swapping $x_{i-1}$ and $x_i$ in 
  $\gamma = x_0 \tensor \cdots \tensor x_d$ equals
  $\gamma$ if $x_{i-1}$ and $x_i$ commute 
  and $-\gamma$ if $x_{i-1}$ and $x_i$ anticommute (hence the choice of
  terminology). Here, note that property (C) yields that 
  $x_{i-1}$ and $x_i$ always commute when they are equal.

  Let $\partial$ be the boundary operator defined on a given
  generator $\gamma = x_0 \tensor \cdots \tensor x_d$ 
  as 
  \begin{equation}
    \partial(\gamma) = \sum_{i=0}^d (-1)^{d-\eta_i} \cdot
    \gamma_i.
    \label{partial-eq}
  \end{equation}
  Here, $\gamma_i$ denotes the element
  $x_0 \tensor \cdots \tensor x_{i-1} \tensor \hat{x_i} \tensor
  x_{i+1} \tensor \cdots \tensor x_d$ obtained by removing the element
  $x_i$, and $\eta_i = \eta_i(\gamma)$ is the number of indices $j
  \in \{i+1, \ldots, d\}$ such that $x_i$ and $x_j$ anticommute.

  We need to show that (\ref{permx-eq}) and (\ref{partial-eq}) are
  consistent for any $(x_0, \ldots, x_d)$. Now, the coefficient of
  $\gamma_j$ in the boundary of 
  $\gamma$ equals $(-1)^{d-\eta_j}$. Let $\gamma'$ be the element
  obtained from $\gamma$ by swapping $x_{i-1}$ and $x_i$. For $j
  \notin \{i-1,i\}$, let  
  $\gamma'_j$ be the element obtained from $\gamma_j$ by swapping
  $x_{i-1}$ and $x_i$. The coefficient of $\gamma'_j$ in the boundary
  of $\gamma'$ remains equal to $(-1)^{d-\eta_j}$, aligning with the
  fact that either $(\gamma',\gamma'_j) = (\gamma,\gamma_j)$ 
  or $(\gamma',\gamma'_j) = (-\gamma,-\gamma_j)$.
  For $j=i$, the coefficient of $\gamma_i$ in the boundary of 
  $\gamma'$ equals $(-1)^{d-\eta_i}$ if and only if 
  $x_{i-1}$ and $x_i$ commute, aligning with the fact that 
  $\gamma' = \gamma$ if and only if $x_{i-1}$ and $x_i$ commute.
  The case $j=i-1$ follows by symmetry.

  \begin{proposition}
    We have that $\mathcal{C}((\Delta,E);R) =
    (C_*((\Delta,E);R),\partial)$ defines a chain
    complex. Equivalently, $\partial \circ \partial = 0$.
  \end{proposition}
  \begin{proof}
    Let $\gamma$ and $\gamma_i$ be defined as above.
    We want to prove that $\partial \circ \partial(\gamma) = 0$. 
    
    It suffices to show that either the coefficient of
    $\gamma_{i,j} = x_0 \tensor \cdots \tensor \hat{x_i} \tensor
    \cdots \tensor \hat{x_j} \tensor \cdots \tensor x_d$ in 
    $\partial \circ \partial(\gamma)$ is zero or
    $\gamma_{i,j}$ itself is zero whenever $0 \le i<j \le d$. 
    We obtain $\gamma_{i,j}$ either by first removing $x_i$ to get
    $\gamma_i$ and then
    removing $x_j$ or by first removing $x_j$ to get $\gamma_j$ and
    then removing $x_i$. By properties (A) and (B), $\gamma_{i,j}$ is
    zero unless each of $\gamma_i$ and $\gamma_j$ is nonzero
    and $x_i$ and $x_j$ anticommute; hence assume 
    these properties are satisfied.

    If we remove first $x_i$ and then $x_j$, then the sign of 
    $\gamma_{i,j}$ equals 
    \[
    (-1)^{d-\eta_i(\gamma)} \cdot (-1)^{d-1-\eta_{j-1}(\gamma_i)}
    =
    (-1)^{2d-1-\eta_i(\gamma)-\eta_{j-1}(\gamma_i)}. 
    \]
    If we proceed the other way
    around, then the sign becomes 
    \[
    (-1)^{d-\eta_j(\gamma)} \cdot(-1)^{d-1-\eta_i(\gamma_j)} 
    = 
    (-1)^{2d-1-\eta_j(\gamma)-\eta_i(\gamma_j)}. 
    \]
    In the first case, we take
    $2d-1$ and subtract from it the number of indices $k>i$ such that
    $x_i$ and $x_k$ anticommute and then the number of indices $k>j$ such
    that $x_j$ and $x_k$ anticommute. In the second case, we again
    take $2d-1$ and subtract from it the number of indices $k>j$ such
    that $x_j$ and $x_k$ anticommute and then the number of indices
    $k>i$, excluding $k=j$, such that $x_i$ and $x_k$ anticommute.
    Thus since $x_i$ and $x_j$ anticommute, the
    two signs cancel out. 
    This concludes the proof.
  \end{proof}

  \section{Properties of the chain complex induced by $(\lambda,{\sf
    s})$}
  \label{lambdascc-sec}

  Let $n \ge 1$, let $\lambda_1, \ldots, \lambda_r$ be positive
  integers summing to $n$, and let ${\sf s} = (s_1, \ldots, s_r)$ be a
  sequence of signs. 
  Let $(U_1,\ldots, U_r)$ be a set partition of $\{1, \ldots, n\}$
  such that 
  $|U_a| = \lambda_a$ for $1 \le a \le r$. 
  In this section, we examine the chain complex obtained
  from that of $\M[n]$ by acting on $\M[n]$ as in 
  (\ref{youngaction-eq}). 
  In this chain complex, it turns out that we may choose generators
  that admit an interpretation as multisets of edges and loops on the
  set $\{1, \ldots, r\}$.
  Specifically, a given matching $\{x_0y_0, \ldots, x_dy_d\} \in
  \M[n]$ corresponds to the multiset $\{a_0b_0, \ldots, a_db_d\}$,
  where $a_i$ and $b_i$ are such that $x_i \in U_{a_i}$ 
  and $y_i \in U_{b_i}$ for $0 \le i \le d$.

  Let $R$ be a commutative ring.
  We consider the action in (\ref{youngaction-eq}).
  To be precise, this is the action by
  the Young group $\Symm{\lambda} =\Symm{\lambda_1} \times \cdots \times
  \Symm{\lambda_r}$ on $\mathcal{C}(\M[n];R)$ given by
  \[
  (\pi,c) \mapsto 
  \prod_{1 \le a \le r : s_a = -1} \sgn(\pi_a) \cdot \pi(c),
  \]
  where $\pi = \pi_1 \cdots \pi_r$, $\pi_i \in \Symm{U_a} \cong
  \Symm{\lambda_a}$ for $1 \le a \le r$, and
  \[
  \pi(x_0y_0 \wedge \cdots \wedge x_dy_d)
  = \pi(x_0)\pi(y_0) \wedge \cdots \wedge \pi(x_d)\pi(y_d).
  \]

  Write
  \[
  \pi^{({\sf s})}(c) = \prod_{1 \le a \le r : s_a = -1} \sgn(\pi_a)
  \cdot \pi(c).
  \]
  For a chain group element $c$, write 
  \[
  {}[c] = \sum_{\pi = \pi_1 \cdots \pi_r \in \Symm{\lambda}} 
  \pi^{({\sf s})}(c).
  \]
  We want to describe the chain subcomplex generated by 
  elements of the form $[c]$. By some abuse of notation, 
  we write this chain complex as $\mathcal{C}(\M[n];R)/(\lambda,{\sf
  s})$.
  For any $\pi \in \Symm{\lambda}$ and chain group element $c$, 
  we note for future reference that
  \begin{equation}
    [\pi^{({\sf s})}(c)] = [c].
    \label{pisig-eq}
  \end{equation}

  It turns out to be helpful to describe generators of 
  $\mathcal{C}(\M[n];R)/(\lambda,{\sf s})$ in terms of multisets
  of edges and loops on the set $\{1, \ldots, r\}$. 
  We refer to an element
  $a \in \{1, \ldots, r\}$ as {\em positively charged} if $s_a = +1$
  and {\em negatively charged} if $s_a = -1$. Two edges $ab$ and $cd$
  {\em commute} if  the intersection $\{a,b\} \cap \{c,d\}$ contains 
  exactly one negatively charged element. Otherwise the two edges {\em
    anticommute}. Every loop anticommutes with all edges and
  loops. As we will see, this aligns with the terminology used in
  Section~\ref{genchain-sec}. 

  For a multiset $\sigma$ of edges and loops on $\{1, \ldots, r\}$ and
  an element  $a \in \{1, \ldots, r\}$, define $\deg_\sigma(a)$ to be
  the {\em degree} of $a$ in 
  $\sigma$; this is the number of times $a$ occurs as a member of an
  edge or a loop in $\sigma$. Note that $a$ occurs twice in the loop
  $aa$, which 
  means that this loop contributes two to the degree of $a$.
  Let $\Delta_{\lambda,{\sf s}}$ be the family of all multisets
  $\sigma$ of edges and loops on the set $\{1,\ldots, r\}$ satisfying
  the following properties. 
  \begin{itemize}
  \item[(i)]
    If the element $a$ is positively charged, then 
    $0 \le \deg_\sigma(a) \le \lambda_a$.
  \item[(ii)] 
    If the element $a$ is negatively charged, then 
    $\lambda_a-1 \le \deg_\sigma(a) \le \lambda_a$,
    and the loop $aa$ does not appear in $\sigma$.
  \item[(iii)]
    An edge $ab$ does not appear more than once in $\sigma$ unless
    $a$ and $b$ have opposite charges. In particular, no loop appears
    more than once in $\sigma$.
  \end{itemize}
  Let $E_{\lambda,{\sf s}}$ be the set consisting of those
  edges and loops $\{ab,cd\}$ on the set 
  $\{ ab : 1 \le a \le b \le r\}$ with the property that 
  $ab$ and $cd$ commute.
  In Section~\ref{genchain-sec}, we needed a total order on the
  underlying set $V$. In the case of $E_{\lambda,{\sf s}}$, the
  set $V$ is the set of edges and loops
  on the set $\{1, \ldots, r\}$. We order the
  elements of $V$ lexicographically; $ab \le cd$ 
  if and only if either $a < c$ or $a = c$ and $b \le d$, where we assume
  that $a \le b$ and $c \le d$ (recall that $ab$ and $ba$ denote the
  same edge).

  Choosing $\Delta = \Delta_{\lambda,{\sf s}}$ and $E = E_{\lambda,{\sf
      s}}$, 
  note that properties (i)-(iii) imply properties (A)-(C) in 
  Section~\ref{genchain-sec}. Namely, each of the families defined by
  (i), (ii), and (iii) satisfies (A); hence $\Delta_{\lambda,{\sf s}}$
  satisfies (A), being the intersection of these families. 
  Moreover, suppose that $\sigma \in \Delta_{\lambda,{\sf s}}$ 
  contains two commuting edges $e$ and $e'$; let $a$ be the unique
  negatively charged element in the intersection of $e$ and
  $e'$. Since the degree of $a$ in $\sigma$ is at most $\lambda_a$,
  the degree of $a$ in $\sigma \setminus \{e,e'\}$ is at most
  $\lambda_a-2$, which implies by (ii) that $\sigma \setminus
  \{e,e'\}$ does not belong to $\Delta_{\lambda,{\sf s}}$; thus (ii)
  implies (B). 
  Finally, by construction, an edge
  commutes with itself if and only if it contains exactly one
  negatively charged element; hence (iii) implies (C). In particular,
  the construction in Section~\ref{genchain-sec} yields a well-defined
  chain complex $\mathcal{C}((\Delta_{\lambda,{\sf s}},E_{\lambda,{\sf
      s}});R)$.

  \begin{theorem}
    Suppose that $|\Symm{\lambda}| = \prod_{a=1}^r \lambda_a!$ is
    not a zero divisor in $R$.
    Then $\mathcal{C}(\M[n];R)/(\lambda,{\sf s})$ and 
    $\mathcal{C}((\Delta_{\lambda,{\sf s}},E_{\lambda,{\sf s}});R)$
    are isomorphic as chain complexes.
    \label{isomorphic-thm}
  \end{theorem}
  \begin{proof}
    Let $\sigma = \{a_0b_0, \ldots, a_db_d\}$ be a multiset on $\{1,
    \ldots, r\}$
    such that $0 \le \deg_\sigma(a) \le \lambda_a$ for all $a \in
    \{1, \ldots, r\}$ and such that $a_0b_0 \le \cdots \le a_db_d$; we
    assume that $a_k \le b_k$ for $0 \le k \le d$.
    Since $|U_a| \ge \deg_\sigma(a)$ for all $a$, there exist
    matchings $\{x_0y_0, \ldots, x_dy_d\}$ satisfying $x_k
    \in U_{a_k}$ and $y_k \in U_{b_k}$ for $0 \le k \le d$.
    Among all such matchings, let $\{x_0y_0, \ldots, x_dy_d\}$
    be the lexicographically smallest matching when viewed as a sequence
    $(x_0,y_0,x_1,y_1, \cdots, x_d,y_d)$. 
    Write
    \[
    \hat{\gamma} = x_0y_0 \wedge \cdots \wedge x_dy_d
    \]
    and $\gamma = [\hat{\gamma}]$,
    and define
    \[
    \varphi(a_0b_0 \tensor \cdots \tensor a_db_d)
    = \gamma.
    \]
    We want to show that $\varphi$ defines an isomorphism from 
    $\mathcal{C}((\Delta_{\lambda,{\sf s}}, E_{\lambda,{\sf s}});R)$
    to $\mathcal{C}(\M[n];R)/(\lambda,{\sf s})$.

    First, note that $\varphi$ would be surjective if we 
    extended its domain to include all elements $a_0b_0 \tensor
    \cdots \tensor a_db_d$ such that 
    $0 \le \deg_\sigma(a) \le \lambda_a$ for all $a \in
    \{1, \ldots, r\}$; every generator of $\mathcal{C}(\M[n];R)$
    corresponds to such a generator.
    In particular, to prove that $\varphi$ is surjective, it suffices
    to prove that $\varphi(a_0b_0 \tensor \cdots \tensor a_db_d)$ is zero 
    whenever $\sigma = \{a_0b_0, \ldots, a_db_d\}$ violates (i), (ii),
    or (iii).
    Since (i) is true by assumption, we may assume that either (ii) or
    (iii) is violated. 
    \begin{itemize}
    \item
      (ii) is violated for some
      negatively charged element $a$; hence the loop $aa$ appears in
      $\sigma$ or $\deg_\sigma(a) \le \lambda_a-2$. 
      \begin{itemize}
      \item
        If the loop $aa$ appears in $\sigma$, 
        then $\hat{\gamma}$ contains an edge $x_iy_i$ such that $x_i, y_i
        \in U_a$. Let $x = x_i$ and $y = y_i$ in this case. 
      \item
        If $\deg_\sigma(a) \le \lambda_a-2$, then 
        there are elements $x,y \in U_a$ such that $x,y \notin
        \{x_i,y_i\}$ for
        all $i$. 
      \end{itemize}
      In both cases, the action of the transposition 
      $(x,y)$ on $\hat{\gamma}$ yields $-\hat{\gamma}$, implying that 
      $2\gamma = 0$ by (\ref{pisig-eq}). Since $|U_a| \ge 2$, we obtain
      that $|\Symm{\lambda}|$ is divisible by $2$. In particular, $2$
      is not a zero divisor in $R$, which yields that $\gamma = 0$.
    \item
      (iii) is violated. This means that
      there are two identical edges  or loops
      $a_ib_i = a_{j}b_{j}$ such that $a_i$ and $b_i$ have the same
      charge.
      Swapping $x_{i}y_{i}$ and $x_jy_j$, we obtain 
      an element $\epsilon$ that is equal to $-\hat{\gamma}$. However,
      $\epsilon$ is also the element obtained by acting on
      $\hat{\gamma}$ with the group element
      $(x_{i},x_j)(y_{i},y_j)$, because $a_i$ and $b_i$ have the
      same charge. We deduce that  
      \[
      -\gamma = [\epsilon] = [\hat{\gamma}] = \gamma; 
      \]
      the second equality is (\ref{pisig-eq}). Again, we obtain that
      $\gamma = 0$.
    \end{itemize}

    To prove that $\varphi$ defines an isomorphism of modules, 
    it remains to show that $\varphi(\kappa)$ is nonzero for every 
    nonzero generator $\kappa$ of 
    $\mathcal{C}((\Delta_{\lambda,{\sf s}},E_{\lambda,{\sf
    s}});R)$. 

    Saying that $\gamma = x_0y_0 \wedge \cdots \wedge x_dy_d$ is zero
    in $\mathcal{C}(\M[n];R)/(\lambda,{\sf s})$ 
    is equivalent to saying that there is a group element $\pi
    \in \Symm{\lambda}$ such that $\pi^{({\sf s})}(\hat{\gamma}) =
    -\hat{\gamma}$. Namely, suppose that such a group element $\pi$
    exists. Then
    \[
    \gamma = [\hat{\gamma}] = [\pi^{({\sf s})}(\hat{\gamma})]
    = [-\hat{\gamma}] = -\gamma,
    \]
    which implies that $\gamma = 0$ by the assumption on $R$.
    If such an element $\pi$ does not exist, then the coefficient of 
    $\hat{\gamma}$ in $\gamma$ is a divisor of $|\Symm{\lambda}|$,
    which is not a zero divisor in $R$ by assumption.

    Let $\pi$ be such that $\pi^{({\sf s})}(\hat{\gamma}) =
    -\hat{\gamma}$. For each $a \le b \in \{1, \ldots, r\}$, 
    let $\pi_{a,b}$ be the restriction of $\pi$ to the union of the
    sets $\{x_i,y_i\}$ satisfying $x_i \in U_a$ and $y_i \in U_b$.
    For each $a \in \{1, \ldots, r\}$, let $\pi_a$ be the restriction
    of $\pi$ to the set of elements $x \in U_a$ such that 
    $x \notin \{x_i,y_i :  i \in \{0, \ldots, d\}\}$.
    We may decompose $\pi$ as the product of all $\pi_{a,b}$ and
    $\pi_a$, extending each factor to $\{1, \ldots, n\}$ by defining
    it to be the
    identity outside its domain. In particular, either 
    $\pi^{({\sf s})}_{a}(\hat{\gamma}) = -\hat{\gamma}$ for some $a$
    or $\pi^{({\sf s})}_{a,b}(\hat{\gamma}) = -\hat{\gamma}$ for some
    $a \le b$.
    In the former case, $a$ is negatively charged, and 
    there are at least two elements $x$ and $y$ in $U_a$ 
    outside the set $\{x_i,y_i :  i \in \{1, \ldots, r\}\}$. This
    violates (ii), as the degree of $a$ is then at most $\lambda_a-2$. 
    In the latter case, $ab$ anticommutes with itself, meaning that
    $a$ and $b$ have the same charge. This violates (iii).

    It remains to show that the boundary operators coincide.
    Consider the $i$th term 
    $\epsilon_i = (-1)^i\cdot x_0y_0 \wedge \cdots \wedge \hat{x_iy_i}
    \wedge x_dy_d$ in the boundary of $\hat{\gamma}$ in
    $\mathcal{C}(\M[N];R)$. 
    Write $\delta_i = [\epsilon_i]$. If $a_i \neq b_i$, then let
    $(z_1 = x_i, z_2, \ldots, z_q)$ be the sequence of elements
    appearing in $U_{a_i} \cap \{x_j,y_j : i \le j \le d\}$
    and let $(z'_1 = y_i, z'_2, \ldots, z'_{q'})$ be the sequence of
    elements appearing in $U_{b_i} \cap \{x_j,y_j : i \le j \le d\}$,
    arranged in increasing order.
    To obtain the lexicographically smallest element $\hat{\delta_i}$
    from $\epsilon_i$, we need to act on $\epsilon_i$ with the
    permutation $(z_q, \cdots, z_2, z_1)(z'_{q'}, \cdots, z'_2,
    z'_1)$. There are three cases.
    \begin{itemize}
    \item
      $a_i$ and $b_i$ are both negatively charged.
      By construction, $a_ib_i$ appears only once in 
      $\gamma$, meaning that there are $q+q'-2$ edges 
      $a_jb_j$ commuting with $a_ib_i$ such that $j>i$;
      hence $\eta_i = d-i-(q+q'-2)$, where
      $\eta_i$ is the number of indices $j > i$ such that $a_ib_i$ and
      $a_jb_j$ anticommute.
      We conclude that the sign of $\hat{\delta_i}$ equals
      $(-1)^{i+q+q'-2} = (-1)^{d-\eta_i}$, which aligns with
      (\ref{partial-eq}).
    \item
      $a_i$ and $b_i$ are both positively charged.
      This means that no edges commute with $a_ib_i$; 
      hence $\eta_i = d-i$ and the sign of 
      $\hat{\delta_i}$ equals $(-1)^{i} = (-1)^{d-\eta_i}$.
    \item
      $a_i$ and $b_i$ have opposite charges. By symmetry, we may
      assume that $a_i$ is negatively charged and $b_i$ positively
      charged. In that case, there are $q-1$ edges 
      $a_jb_j$ commuting with $a_ib_i$ such that $j>i$;
      hence $\eta_i = d-i-(q-1)$. We conclude that
      the sign of $\hat{\delta_i}$ equals
      $(-1)^{i+q-1} = (-1)^{d-\eta_i}$. 
    \end{itemize}
    If $a_i = b_i$, then $a_i$ is positively charged by construction, 
    and there are no edges or loops commuting with $a_ia_i$;
    hence $\eta_i = d-i$.
    We conclude that the sign of $\hat{\delta_i}$ is 
    $(-1)^i = (-1)^{d -\eta_i}$.
  \end{proof}

  When all signs $s_i$ are positive, we may 
  identify $\mathcal{C}(\M[n];R)/(\lambda,{\sf s})$ with the
  chain complex of a certain 
  simplicial complex. 
  Specifically, define $\BD{r}{\lambda}$ to be
  the family of sets $\sigma$ of edges and loops on the set
  $\{1, \ldots, r\}$ such that $\deg_\sigma(a) \le \lambda_a$ for 
  each $a \in \{1, \ldots, r\}$. It is clear that $\BD{r}{\lambda}$ is
  a simplicial complex.
  \begin{proposition}
    If $s_a = +1$ for all $a \in \{1, \ldots, r\}$, then
    $\mathcal{C}((\Delta_{\lambda,{\sf s}},E_{\lambda,{\sf s}});R)$
    and $\mathcal{C}(\BD{r}{\lambda};R)$ are isomorphic.
    In particular, if $|\Symm{\lambda}| = \prod_{a=1}^r \lambda_a!$ is
    not a zero divisor in $R$, then
    $\mathcal{C}(\M[n];R)/(\lambda,{\sf s})$ and 
    $\mathcal{C}(\BD{r}{\lambda};R)$ are isomorphic.
  \end{proposition}
  \begin{proof}
    By assumption, all elements are positively charged, which 
    implies by (i)-(iii) that the individual chain groups are 
    isomorphic for the two chain complexes. Since all pairs of edges 
    anticommute, the boundary operator $\partial$ 
    on $\mathcal{C}((\Delta_{\lambda,{\sf s}},E_{\lambda,{\sf s}});R)$ defined in
    (\ref{partial-eq}) has the property that 
    \[
    \partial(\gamma) = \sum_{i=0}^d (-1)^{i} \cdot \gamma_i,
    \]
    which is the usual simplicial boundary operator. For the last
    statement of the proposition, apply Theorem~\ref{isomorphic-thm}.
  \end{proof}

  It turns out that we can transform
  $\mathcal{C}(\M[n];R)/(\lambda,{\sf s})$ into a smaller chain
  complex with the same homology such that 
  the description of the generators is symmetric with respect to
  charge.
  \begin{theorem}    
    The homology of $\mathcal{C}((\Delta_{\lambda,{\sf
        s}},E_{\lambda,{\sf s}});R)$
    is isomorphic to the homology of the quotient
    complex $\mathcal{C}((\lambda,{\sf s});R)$ 
    obtained by restricting
    to multisets $\sigma$ satisfying the following two conditions.
    \begin{itemize}
    \item[\rm(I)]
      For each element $a$, we have that
      $\lambda_a-1 \le \deg_\sigma(a) \le \lambda_a$, 
      and the loop $aa$ does not appear in $\sigma$.
    \item[\rm(II)]
      An edge $ab$ does not appear more than once in $\sigma$ unless
      $a$ and $b$ have opposite charges.
    \end{itemize}
    In particular, if $|\Symm{\lambda}| = \prod_{a=1}^r \lambda_a!$ is
    not a zero divisor in $R$, then 
    the homology of $\mathcal{C}(\M[n];R)/(\lambda,{\sf s})$
    isomorphic to that of $\mathcal{C}((\lambda,{\sf s});R)$.
    \label{corechaincpx-thm}
  \end{theorem}
  \begin{proof}
    Let $a_1, \ldots, a_p$ be the positively charged elements. 
    For $0 \le q \le p$, let $\mathcal{C}^{(q)}$ be the quotient
    complex 
    obtained by removing all generators in which $a_i$ does not
    satisfy (I) for some $i \le q$. Note that $\mathcal{C}^{(0)}$
    coincides with $\mathcal{C}((\Delta_{\lambda,{\sf
        s}},E_{\lambda,{\sf s}});R)$
    and that $\mathcal{C}^{(p)}$ coincides with 
    $\mathcal{C}((\lambda,{\sf s});R)$.
    To prove the theorem, it suffices to prove that
    the homology of $\mathcal{C}^{(q-1)}$ is isomorphic to that of 
    $\mathcal{C}^{(q)}$ for $1 \le q \le p$ and that an isomorphism is
    induced by the natural quotient map. Namely, this implies that 
    the natural quotient map from $\mathcal{C}((\Delta_{\lambda,{\sf
        s}},E_{\lambda,{\sf s}});R)$ to $\mathcal{C}((\lambda,{\sf
      s});R)$ induces an isomorphism on homology. 

    Now,
    note that we may view $\mathcal{C}^{(q)}$ as the quotient complex
    of $\mathcal{C}^{(q-1)}$ by the subcomplex
    $\mathcal{W}^{(q)}$ consisting of those generators in which $a_q$
    does not satisfy (I). Given a chain group element $c$ in 
    $\mathcal{W}^{(q)}$, we may decompose $c$ as
    $c = a_qa_q \tensor c' +  c''$, where no generators appearing
    in $c'$ and $c''$ contain the loop $a_qa_q$. This means that the
    degree of $a_q$ is at most $\lambda_{a_q}-2$ in each generator
    appearing in $c'$ or $c''$. 
    Let $\partial$ denote the boundary operator in $\mathcal{W}^{(q)}$
    induced by the one in $\mathcal{C}((\Delta_{\lambda,{\sf
        s}},E_{\lambda,{\sf s}});R)$. 
    Since
    \[
    \partial(c) = c' - a_qa_q \tensor \partial(c') +  \partial(c''),
    \]
    $c$ is a cycle if and only if $\partial(c') = 0$
    and $c' = - \partial(c'')$. However, in this case we have that
    $c$ equals $\partial(a_qa_q \tensor c'')$. This is indeed a
    boundary in $\mathcal{W}^{(i)}$, because the degree of $a_q$ is at
    most $\lambda_{a_q}-2$ in each generator appearing in $c''$.
    As a consequence, the homology of $\mathcal{W}^{(i)}$ is zero. 
    By the long exact sequence for the pair
    $(\mathcal{C}^{(i-1)},\mathcal{W}^{(i)})$, it follows that 
    $\mathcal{C}^{(i-1)}$ and $\mathcal{C}^{(i)}$ have the same
    homology. This concludes the proof.
  \end{proof}

  \section{Connection to the free two-step nilpotent Lie algebra}
  \label{twosteplie-sec}

  Before proceeding, we discuss a closely related Koszul complex
  \cite{JW,Sigg}.
  Let $R$ be a commutative ring.
  Let $\lambda = (\lambda_1, \ldots, \lambda_r)$ and
  \[
  \lambda^- = 
  \left(\begin{array}{ccc} {\lambda_1} &
      {\cdots} & {\lambda_r} \\
      {-} & {\cdots} & {-}
    \end{array}
  \right).
  \]
  Let $X_{d+1}^\lambda$ be the free $R$-module generated by 
  elements of the form $e_0 \wedge \cdots \wedge e_d \tensor v_1
  \wedge \cdots \wedge v_t$, where $v_1, \ldots, v_t \in \{1, \ldots,
  r\}$
  and 
  $e_0, \ldots, e_d$ are edges on $\{1, \ldots, r\}$ (loops are not
  allowed) such that the total number of occurrences of $a$ in $(e_0,
  \ldots, e_d, v_1, \ldots, v_t)$ is $\lambda_a$ for $1 \le a \le r$.
  If $e_i = e_j$ or $v_i = v_j$ for some $i \neq j$, then we define 
  $e_0 \wedge \cdots \wedge e_d \tensor v_1
  \wedge \cdots \wedge v_t$ to be zero.
  For any permutations $\rho$ of $\{0, \ldots, d\}$ and $\tau$ of
  $\{1, \ldots, t\}$, we define 
  \[
  e_{\rho(0)} \wedge \cdots \wedge e_{\rho(d)} \tensor v_{\tau(1)}
  \wedge \cdots \wedge v_{\tau(t)} 
  = 
  \sgn(\rho)\sgn(\tau) \cdot e_0 \wedge \cdots \wedge e_d
  \tensor v_1 \wedge \cdots \wedge v_t.
  \]
  Note that $2d+2+t = |\lambda|$.
  For example, for $\lambda = (2,2,2,1)$, the element 
  $12 \wedge 23 \tensor 1 \wedge 3 \wedge 4$ appears in
  $X_2^\lambda$.

  We define a boundary operator by
  \begin{eqnarray*}
    & & \delta
    (x_0y_0 \wedge \cdots \wedge x_dy_d \tensor v_1 \wedge
    \cdots \wedge v_t)\\
    &=& \sum_{i=0}^d 
    (-1)^{i} \cdot x_0y_0 \wedge \cdots \wedge \hat{x_iy_i} \wedge \cdots
    \wedge x_dy_d \tensor x_i \wedge y_i \wedge v_1 \wedge \cdots
    \wedge
    v_t;
  \end{eqnarray*}
  we assume that $x_j < y_j$ for $0 \le j \le d$.
  This indeed yields a chain complex $\mathcal{X}^\lambda =
  (X_*^\lambda,\delta)$; what we just defined is the chain complex
  associated to the free two-step 
  nilpotent Lie algebra analyzed by J\'ozefiak
  and Weyman \cite{JW} and Sigg \cite{Sigg}.

  \newcommand{\inv}[1]{\mbox{inv}(#1)}

  For a sequence $w = (w_1, \ldots, w_k)$ of elements from $\{1,
  \ldots, r\}$, 
  let $\inv{w}$ be the number of inversions of $w$, i.e., the number
  of pairs $(i,j)$ such that $i<j$ and $w_i>w_j$. 
  For a generator $\gamma = x_0y_0 \wedge \cdots \wedge x_dy_d
  \tensor v_1 \wedge \cdots \wedge v_t$, let 
  \[
  \langle \gamma \rangle 
  = (x_0,y_0, \ldots, x_d, y_d,  v_1, \ldots, v_t).
  \]
  Define 
  $\sgn(\gamma)$ to be $(-1)^{\inv{\langle\gamma\rangle}}$.
  For example, 
  \[
  \sgn(12 \wedge 23 \tensor 1 \wedge 3 \wedge 4) = -1,
  \]
  because $(1,2,2,3,1,3,4)$ contains $3$ inversions: $(2,5), (3,5),
  (4,5)$.
  
  Define a map $\varphi$ from $X_{d+1}^\lambda$ to 
  $\tilde{C}_d(\lambda^-;R)$ by
  \[
  \varphi(\gamma) = \varphi(e_0 \wedge \cdots \wedge e_d \tensor v_1
  \wedge \cdots \wedge v_t) = 
  \sgn(\gamma) \cdot e_0 \tensor \cdots \tensor e_d.
  \]
  We need to show that this map is well-defined. For this, it suffices
  to show that $\varphi(\gamma') = - \varphi(\gamma)$ for any
  $\gamma'$ obtained from $\gamma$ by swapping either $e_{i-1}$ and
  $e_i$ or $v_{i-1}$ and $v_i$ for some $i$.
  \begin{itemize}
  \item
    We obtain $\gamma'$ by swapping $e_{i-1}$ and $e_i$. 
    \begin{itemize}
    \item
      If the two edges are disjoint, then the number of inversions changes
      by an even number, because the inversion status changes for exactly
      four pairs of indices. In particular, $\sgn(\gamma') =
      \sgn(\gamma)$. Since $e_{i-1}$ and $e_i$ anticommute, we obtain
      that $\varphi(\gamma') = -\varphi(\gamma)$. 
    \item
      If the two edges have one element in common, then the number of
      inversions changes by an odd number, because there are three pairs
      of indices for which the inversion status changes. 
      This means that $\sgn(\gamma') = -\sgn(\gamma)$. 
      Since $e_{i-1}$ and $e_i$ commute, we again deduce that
      $\varphi(\gamma') = -\varphi(\gamma)$.
    \end{itemize}
  \item
    We obtain $\gamma'$ by swapping $v_{i-1}$ and $v_i$. 
    Since the number of inversions either increases or decreases by one,
    we obtain that $\sgn(\gamma') = -\sgn(\gamma)$, which implies that
    $\varphi(\gamma') = -\varphi(\gamma)$.
  \end{itemize}

  \begin{theorem}
    The map $\varphi$ defines an isomorphism between
    $\mathcal{X}^\lambda$ and
    $\mathcal{C}(\lambda^-;R)$.
  \end{theorem}
  \begin{proof}
    It is straightforward to check that we have an
    isomorphism between the individual chain groups. 
    Consider a generator 
    \[
    \gamma = e_0 \wedge \cdots \wedge e_d \tensor v_1 \wedge \cdots
    \wedge v_t, 
    \]
    and define
    \[
    \gamma_i = e_0 \wedge \cdots \wedge \hat{e_i} \wedge \cdots
    \wedge e_d \tensor x_i \wedge y_i \wedge v_1 \wedge \cdots
    \wedge v_t,    
    \]
    where $e_i = x_iy_i$.
    The coefficient of $\gamma_i$ in the boundary of $\gamma$ in
    $\mathcal{X}^\lambda$ equals $(-1)^i$,
    which implies that the coefficient of 
    $e_0 \tensor \cdots \tensor \hat{e_i} \tensor \cdots \tensor e_d$
    in  $\varphi\circ \delta(\gamma)$ equals 
    $\sgn(\gamma_i) \cdot (-1)^i$.
    Moreover, the coefficient of the same generator in 
    $\partial\circ \varphi(\gamma)$ equals $(-1)^{d-\eta_i}\cdot
    \sgn(\gamma)$,
    where $\eta_i$ is the number of edges among $e_{i+1},
    \ldots, e_d$ that are disjoint from $e_i$.
    For these two coefficients to be the same, we need that
    \begin{equation}
      \sgn(\gamma_i)= (-1)^{d-\eta_i} \cdot (-1)^i \cdot \sgn(\gamma).
      \label{sgngamma-eq}
    \end{equation}
    Now, we may get from $\langle \gamma \rangle$ to 
    $\langle \gamma_i \rangle$ by applying $4(d-i)$ transpositions
    of adjacent elements; we first move $y_i$ to its new position via
    $2(d-i)$  transpositions and then move $x_i$ via the same
    number of transpositions.
    The number of inversions increases or decreases by one whenever we
    transpose two distinct elements and stays the same
    otherwise. Since all edges are distinct, the total number 
    of times the number of inversions changes is equal to $4(d-i)$
    minus the number of edges among $e_{i+1}, \ldots, e_d$ that are
    not disjoint from $e_i$. This equals
    \[
    4(d-i)-((d-i)-\eta_i) = 3d-3i + \eta_i \equiv i + d - \eta_i
    \pmod{2}, 
    \]
    which implies (\ref{sgngamma-eq}).
  \end{proof}
  For $R = \Q$, Dong and Wachs \cite{DongWachs} showed
  that the homology of
  $\mathcal{X}^\lambda$ is isomorphic to that of
  $\mathcal{C}(\lambda^+;\Q)$, where 
  $\lambda^+$ is obtained from $\lambda^-$ by replacing all minus
  signs with plus signs. As it turns out, this property
  does not remain true if we replace $\Q$ with $\Z$.  
  Namely, in general, the homology of
  $\mathcal{C}(\lambda^+;\Z)$ is not 
  isomorphic to the homology of
  $\mathcal{C}(\lambda^-;\Z)$. 
  The smallest example is given by $\lambda = (1,2,2,2)$, in which
  case we obtain elements of order $3$ for the natural signed action
  but not for the natural unsigned action. Computational evidence
  suggests that the torsion subgroup of the homology of
  $\mathcal{C}(\lambda^+;\Z)$ 
  tends to be smaller than the torsion subgroup of the homology of
  $\mathcal{C}(\lambda^-;\Z)$, but we have not been able to make this
  observation more precise, let alone prove anything in this
  direction. 
  
  All results listed in
  Tables~\ref{bdtwist5even-tab}-\ref{bdtwist7even-tab} can be
  interpreted as results about the homology of $\mathcal{X}^\lambda$
  for various $\lambda$; all signs $s_i$ are $-1$. 
  For example, the homology of each of
  $\mathcal{X}^{(3,3,3,3,3,3)}$, $\mathcal{X}^{(3,3,3,3,3,3,3)}$, 
  $\mathcal{X}^{(3,3,3,3,3,3,3,3)}$, and 
  $\mathcal{X}^{(4,4,4,4,4,4,4)}$ contains elements of order $5$. 
  Keep in mind that all degrees should be shifted one step up.

  \section{Actions induced by direct products of wreath products}
  \label{wreath-sec}

  In some situations where the complex $\mathcal{C}((\lambda,{\sf
  s});R)$ remains too large to admit a direct analysis, 
  we will consider a group action on $\mathcal{C}((\lambda,{\sf
  s});R)$. 
  Specifically, let $T = (T_1, \ldots, T_q)$ be a partition of $\{1,
  \ldots, r\}$
  such that $(\lambda_a,s_a) = (\lambda_b,s_b)$ for all  
  $a$ and $b$ in the same set $T_k$. The 
  Young group $\Symm{T} = \Symm{T_1} \times \cdots \times
  \Symm{T_q}$ acts on $\mathcal{C}((\lambda,{\sf s});R)$
  by permuting the elements in each $T_k$. 
  In this paper, we will only
  consider the natural unsigned action. 

  Assume that $|\Symm{\lambda}|$ and $|\Symm{T}|$ are not zero
  divisors in $R$.
  Note that we obtain the chain complex
  $\mathcal{C}((\lambda,{\sf s});R)/\Symm{T}$ from
  $\mathcal{C}(\M[n];R)$ in three steps. 
  \begin{itemize}
  \item[1.]
    Let $\Symm{\lambda}$ act on $\mathcal{C}(\M[n];R)$ to form 
    $\mathcal{C}(\M[n];R)/(\lambda,{\sf s})$. 
  \item[2.]
    Transform
    $\mathcal{C}(\M[n];R)/(\lambda,{\sf s}) \cong 
    \mathcal{C}((\Delta_{\lambda,{\sf s}},E_{\lambda,{\sf s}});R)$ 
    into $\mathcal{C}((\lambda,{\sf s});R)$ via the natural quotient
    map. 
  \item[3.]
    Let $\Symm{T}$ act on
    $\mathcal{C}((\lambda,{\sf s});R)$ to form 
    $\mathcal{C}((\lambda,{\sf s});R)/\Symm{T}$.
  \end{itemize}
  One may interchange steps 2 and 3, thus first letting the group
  $\Symm{T}$ act on $\mathcal{C}(\M[n];R)/(\lambda,{\sf s})$ 
  and then forming a quotient complex satisfying properties I-II in
  Theorem~\ref{corechaincpx-thm}. 
  In this manner, we may view 
  $(\mathcal{C}((\lambda,{\sf s});R)/\Symm{T}$ as a quotient
  complex of 
  $\tilde{C}_d(\M[n];\Z)/G(\lambda,{\sf s},T)$, 
  where $G(\lambda,{\sf s},T)$ is a direct product of certain wreath
  products defined in terms of $\Symm{\lambda}$ and $\Symm{T}$. 

  Specifically, define $\mu_k = |T_k|$. For simplicity, assume that
  $k \in T_k$ for $1 \le k \le q$ and that $\mu_k = 1$ if and only if
  $k$ is greater than a given value $p$. Then $G(\lambda,{\sf s},T)$
  is isomorphic to 
  \[
  (\Symm{\lambda_1} \wr \Symm{\mu_1}) \times
  \cdots \times (\Symm{\lambda_p} \wr \Symm{\mu_p}) \times
  \Symm{\lambda_{p+1}} \times \cdots \times \Symm{\lambda_q},
  \]
  where $\Symm{\lambda_k} \wr \Symm{\mu_k}$ denotes the wreath product
  of $\Symm{\lambda_k}$ by $\Symm{\mu_k}$. We represent the action as
  \begin{equation}
  \left\{
    \mu_1 \times 
    \begin{array}{c}
      \qw{\lambda_1}\\ 
      \qw{s_1}
    \end{array}\right\}
  \join
  \left\{
    \mu_2 \times 
    \begin{array}{c}
      \qw{\lambda_2}\\ 
      \qw{s_2}
    \end{array}\right\}
  \join
  \cdots 
  \join
  \left\{
    \mu_p \times 
    \begin{array}{c}
      \qw{\lambda_p}\\ 
      \qw{s_p}
    \end{array}\right\}
  \join
    \left(\begin{array}{cccc} {\lambda_{p+1}} & {\lambda_{p+2}} &
        {\cdots} & {\lambda_q}
        \\ 
        {s_{p+1}} & {s_{p+2}} & {\cdots} & {s_q}
      \end{array}\right).
    \label{mulambdasnot-eq}
  \end{equation}

  Let $\mathcal{W}$ be the subcomplex of 
  $\mathcal{C}(\M[n];R)/(\lambda,{\sf s})$ with the
  property 
  that the complex $\mathcal{C}((\lambda,{\sf s});R)$ is the quotient
  of $\mathcal{C}(\M[n];R)/(\lambda,{\sf s})$ by $\mathcal{W}$. 
  Since $\mathcal{W}$ is acyclic by the proof of
  Theorem~\ref{corechaincpx-thm}, the homology of
  $\mathcal{W}/\Symm{T}$ has exponent dividing $|\Symm{T}|$ and is
  hence finite; apply Proposition~\ref{qpreserved-prop}.
  By the exact sequence
  \[
  \begin{CD}
    \tilde{H}_i(\mathcal{W}/\Symm{T})
    @>>>
    \tilde{H}_i((\mathcal{C}(\M[n];R)/(\lambda,{\sf s}))/\Symm{T}) \\
    @>>>
    \tilde{H}_i(\mathcal{C}((\lambda,{\sf s});R)/\Symm{T})
    @>>>
    \tilde{H}_{i-1}(\mathcal{W}/\Symm{T}),
  \end{CD}
  \]
  we may hence deduce that 
  $\tilde{H}_i((\mathcal{C}(\M[n];R)/(\lambda,{\sf s}))/\Symm{T})$
  and 
  $\tilde{H}_i(\mathcal{C}((\lambda,{\sf s});R)/\Symm{T})$ 
  nearly coincide, the exception being that the Sylow $p$-subgroups
  of their torsion parts
  may differ when $p$ divides $|\Symm{T}|$.

  If some block $T_j$ of size at least two consists of
  positively charged elements, then the homology of
  $\mathcal{W}/\Symm{T}$ is indeed not necessarily zero. For example,
  if
  $\lambda = \left(\begin{array}{ccc} \qw{3} &
      \qw{3} & \qw{3} \\
      \qw{+} & \qw{+} & \qw{+} 
    \end{array}\right)$ and $\Symm{T} = \Symm{3}$, then
  the homology in degree two is a group of order two.

  However, if each block $T_j$ of size at least two consists of
  negatively 
  charged elements, then the homology groups do coincide;
  it is not hard to adapt the proof of Theorem~\ref{corechaincpx-thm}
  to prove that $\mathcal{W}/\Symm{T}$ has vanishing homology.
  All examples considered in this paper are of this type.

  The problem of finding a combinatorial description of the generators
  of $\mathcal{C}((\lambda,{\sf s});R)/\Symm{T}$
  appears to be immensely hard in general; we are not aware of any
  simple characterization similar to the one in 
  Theorem~\ref{corechaincpx-thm}. 

  \section{Detecting elements of order $5$ in the homology of $\M[14]$}
  \label{m14-sec}

  We present a computer-free proof that $\tilde{H}_4(\M[14];\Z)$
  contains elements of order $5$.  The proof consists of two steps. In
  the first step we consider the action on $\M[14]$ induced by
  $\left(\begin{array}{ccccc} \qw{2} & \qw{3} & \qw{3} & \qw{3} &  
      \qw{3}
      \\[-1ex] \qw{+} & \qw{-} & \qw{-} & \qw{-} & \qw{-}
    \end{array}\right)$.
  In the second step we proceed with the unsigned action induced by
  the natural action by $\Symm{4}$ on the four parts of size three,
  thus analyzing the action induced by
          $
          \left(
            \begin{array}{c}
              \qw{2}\\ 
              \qw{+}
            \end{array}\right)
          \join
          \left\{ 4 \times \begin{array}{c}
                \qw{3}\\ 
              \qw{-} 
            \end{array}\right\}
          $.
  The resulting chain complex consists of two free groups of
  rank four, one in degree four and one in degree five. An explicit
  calculation of the boundary map yields a matrix with determinant 
  $\pm 5$.

  Unfortunately, the proof does not shed much light on {\em why} we
  end up with the value $5$. In particular, we do not know how to
  generalize the proof to deduce the existence of elements of
  order $(2r-1)$
  in the homology of $\M[(r+1)^2-2]$ for general $r$. We expect that
  this can be done by analyzing the action induced by  
  $\left(
  \begin{array}{c}
    \qw{r-1}\\ 
    \qw{+}
  \end{array}\right)
  \join
  \left\{ (r+1) \times \begin{array}{c}
    \qw{r}\\ 
    \qw{-} 
  \end{array}\right\}$; recall the discussion after
  Theorem~\ref{7tor-thm} in Section~\ref{intro-sec} and in
  Section~\ref{big-sec}.

  Thus pick the action on $\M[14]$ induced by 
  $\left(\begin{array}{c} \lambda
      \\[-1ex] {\sf s} \end{array}\right) = 
  \left(\begin{array}{ccccc} \qw{2} & \qw{3} & \qw{3} & \qw{3} & 
      \qw{3}
      \\[-1ex] \qw{+} & \qw{-} & \qw{-} & \qw{-} & \qw{-}
    \end{array}\right)$. By Theorem~\ref{corechaincpx-thm}, the
  homology of the resulting chain complex
  $\mathcal{C}(\M[14];\Z)/(\lambda,{\sf s})$ is isomorphic to the homology
  of the chain complex $\mathcal{C}((\lambda,{\sf s});\Z)$.
  In this complex, we have one generator for each  
  loopless multigraph on the vertex set $\{o,1,2,3,4\}$ such that the 
  degree of $o$ is $1$ or $2$, the degrees of the other vertices are $2$ 
  or $3$, and there are no multiple edges between vertices in
  $\{1,2,3,4\}$; we identify each multigraph wih its multiset of edges.
  A careful examination yields that we have one generator for each 
  of the following multigraphs, where $\{a,b,c,d\} = \{1,2,3,4\}$.

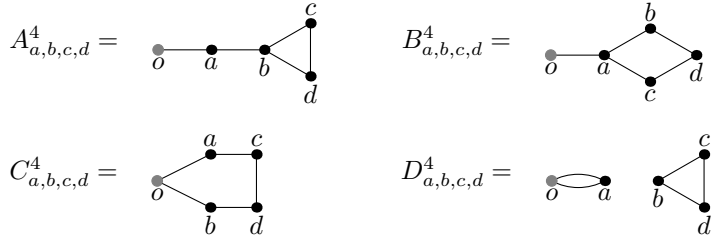
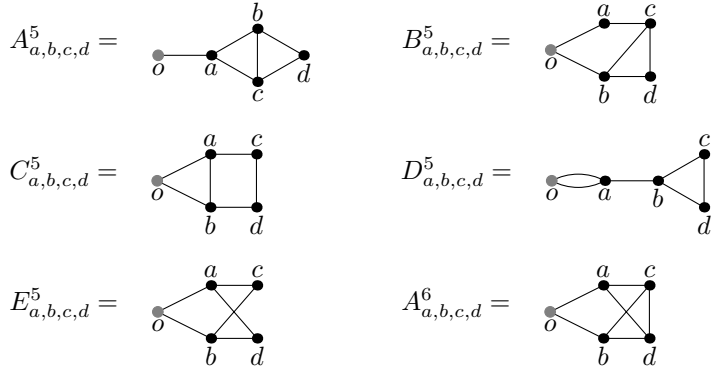
  \begin{tabular}{ccccccc}
    & 
    \multirow{4}{2.5cm}{\epsfig{file=A4.eps}} 
    & 
    & 
    &
    \multirow{4}{2.5cm}{\epsfig{file=B4.eps}}
    \\
    $A^4_{a,b,c,d} =$
    & 
    & 
    & 
    $B^4_{a,b,c,d} =$
    &
    \\
    \\
    \\
  \end{tabular}

  \begin{tabular}{ccccccc}
    & 
    \multirow{4}{2.5cm}{\epsfig{file=C4.eps}} 
    & 
    & 
    &
    \multirow{4}{2.5cm}{\epsfig{file=D4.eps}}
    \\
    $C^4_{a,b,c,d} =$
    & 
    & 
    & 
    $D^4_{a,b,c,d} =$
    &
    \\
    \\
    \\
  \end{tabular}

  \begin{tabular}{ccccccc}
    & 
    \multirow{4}{2.5cm}{\epsfig{file=F5.eps}} 
    & 
    & 
    &
    \multirow{4}{2.5cm}{\epsfig{file=G5.eps}}
    \\
    $A^5_{a,b,c,d} =$
    & 
    & 
    & 
    $B^5_{a,b,c,d} =$
    &
    \\
    \\
    \\
  \end{tabular}

  \begin{tabular}{ccccccc}
    & 
    \multirow{4}{2.5cm}{\epsfig{file=E5.eps}} 
    & 
    & 
    &
    \multirow{4}{2.5cm}{\epsfig{file=H5.eps}}
    \\
    $C^5_{a,b,c,d} =$
    & 
    & 
    & 
    $D^5_{a,b,c,d} =$
    &
    \\
    \\
    \\
  \end{tabular}

  \begin{tabular}{ccccccc}
    & 
    \multirow{4}{2.5cm}{\epsfig{file=I5.eps}} 
    & 
    & 
    &
    \multirow{4}{2.5cm}{\epsfig{file=A6.eps}}
    \\
    $E^5_{a,b,c,d} =$
    & 
    & 
    & 
    $A^6_{a,b,c,d} =$
    &
    \\
    \\
    \\
  \end{tabular}

  Arrange the set of edges on $\{o,1,2,3,4\}$
  lexicographically according to the order $o < 1 < 2 < 3 < 4$;
  \[
  o1 < o2 < o3 < o4 < 12 < 13 < 14 < 23 < 24 < 34.
  \]
  Consider the unsigned action on $\mathcal{C}((\lambda,{\sf s});\Z)$ by
  $\Symm{4}$ induced by the natural action on the set
  $\{1,2,3,4\}$. In the resulting chain complex $\mathcal{C}'$, we
  have one generator for each of 
  the above isomorphism classes of multigraphs. 

  Yet, note that 
  $E^5_{1,2,3,4} = E^5_{2,1,3,4} = \{o1,o2,13,14,23,24\}$. 
  Writing $\gamma = o1 \tensor o2 \tensor 13 \tensor 14 \tensor 23
  \tensor 24$, we have that
  \[
  (1,2)\circ \gamma
  = o2 \tensor o1 \tensor 23 \tensor 24 \tensor 13 \tensor 14
  = - o1 \tensor o2 \tensor 13 \tensor 14 \tensor 23 \tensor 24
  = - \gamma;
  \]
  this is because the sequence $(o2, o1, 23, 24, 13, 14)$
  contains three anticommuting inversions 
  $(1,2)$, $(3,6)$, and $(4,5)$. In particular, 
  the generator corresponding to the isomorphism class of
  $E^5_{1,2,3,4}$ is zero.
  For similar reasons, the generator corresponding to
  the isomorphism class of $A^6_{1,2,3,4}$ is also zero. 

  The remaining eight generators are nonzero though. For example, 
  the identity and $(1,2)(3,4)$ are the two group elements in 
  $\Symm{4}$ that leave $C^4_{1,2,3,4}$ fixed. Since
  \[
  (1,2)(3,4) \circ (o1 \tensor o2 \tensor 13 \tensor 24 \tensor 34)
  = o2 \tensor o1 \tensor 24 \tensor 13 \tensor 34,
  \]
  and since the sequence $(o2,o1,24,13,34)$ contains two anticommuting
  inversions $(1,2)$ and $(3,4)$, we obtain the desired
  claim. The seven other generators are treated similarly.

  For any multigraph $G$ on the vertex set $\{o,1,2,3,4\}$, fix the
  orientation given by the lexicographic order defined earlier.
  For example, we identify $B^4_{3,1,4,2} = \{o3,13,34,12,24\}$
  with the oriented simplex $o3 \tensor 12 \tensor 13 \tensor 24
  \tensor 34$. 
  Identify each generator $G$ with its isomorphism class.
  Moreover, write $A^4 = A^4_{1,2,3,4}$ and so on. 
  We obtain the following.
  \begin{eqnarray*}
    \partial(A^5) &=& A^4_{1,2,3,4} + A^4_{1,3,2,4} - B^4_{1,2,3,4}
    = 2A^4 - B^4; \\
    \partial(B^5) &=& -A^4_{1,3,2,4} - C^4_{1,2,3,4}
    = -A^4 - C^4; \\
    \partial(C^5) &=& -B^4_{1,2,3,4} + B^4_{2,1,4,3} + C^4_{1,2,3,4}
    = -2B^4 + C^4; \\
    \partial(D^5) &=&  A^4_{1,2,3,4} + A^4_{1,2,3,4} + D^4_{1,2,3,4}
    =  2A^4 + D^4. \\
  \end{eqnarray*}
  In matrix form, we get
  \[
  \partial
  \left( 
    \begin{array}{c}
      A^5 \\ B^5 \\ C^5 \\ D^5
    \end{array}
  \right)
  =
  \left( 
    \begin{array}{rrrr}
       2 & -1 &  0 & 0 \\
      -1 &  0 & -1 & 0 \\
       0 & -2 &  1 & 0 \\
       2 &  0 &  0 & 1 \\
    \end{array}
  \right)
  \left( 
    \begin{array}{c}
      A^4 \\ B^4 \\ C^4 \\ D^4
    \end{array}
  \right)
  \]
  Since the determinant is $-5$, we conclude that 
  $\tilde{H}_4(\mathcal{C}') \cong \Z_5$. As a consequence, since 
  the order of $\Symm{4}$ is not divisible by five, 
  $\tilde{H}_4((\lambda,{\sf s});\Z)$ contains elements of order five,
  as does $\tilde{H}_4(\M[14];\Z)$ for similar reasons.

  \section{The case of an abelian $2$-group}
  \label{abeliantwo-sec}

  Let us discuss the special case that all values in the sequence
  $\lambda$ are at most two. Before considering the matching complex,
  we look at a more general situation.
  Let $R$ be a commutative ring such that $2$ is a unit. Let 
  \[
  \begin{CD}
    \mathcal{C} : \cdots @>\partial>> C_{d+1} @>\partial>> C_d
    @>\partial>> C_{d-1} @>\partial>> \cdots 
  \end{CD}
  \]
  be a chain complex of $R$-modules. Write $C = \bigoplus_d
  C_d$. 
  Suppose that $\tau$ is an involution on $\mathcal{C}$; 
  $\tau$ generates a group of size $2$ acting on $\mathcal{C}$.
  Consider the subgroups $C_d^+$ and $C_d^-$ of $C_d$ induced by the
  unsigned and signed actions on $\mathcal{C}$, respectively;
  \begin{eqnarray*}
    C_d^+ &=& \{c+\tau(c) : c \in C_d\}; \\
    C_d^- &=& \{c-\tau(c) : c \in C_d\}.
  \end{eqnarray*}
  We obtain two chain complexes $\mathcal{C}^+$ and $\mathcal{C}^-$.
  Write $C^+ = \bigoplus C^+_d$ and $C^- = \bigoplus C^-_d$.
  \begin{proposition}
    With notation as above, we have that 
    \[
    \mathcal{C} = \mathcal{C}^+ \oplus \mathcal{C}^-.
    \]
    In particular, 
    \[
    H_d(\mathcal{C}) \cong H_d(\mathcal{C}^+) \oplus H_d(\mathcal{C}^-).
    \]
    \label{formula-prop}
  \end{proposition}
  \begin{proof}
    \newcommand{\Id}{\mbox{Id}}
    Note that we may decompose the identity as a sum of two orthogonal
    idempotent endomorphisms; we have that 
    \[
    \Id = \frac{1}{2}(\Id + \tau) + \frac{1}{2}(\Id - \tau). 
    \]
    Since  
    $\frac{1}{2}(\Id + \tau)(\mathcal{C}) = \mathcal{C}^+$
    and 
    $\frac{1}{2}(\Id - \tau)(\mathcal{C}) = \mathcal{C}^-$,
    we are done.
  \end{proof}
    
  More generally, assume that we have $r$ pairwise commuting
  involutions $\tau_1, \ldots, \tau_r$ and hence an elementary abelian
  $2$-group of order $2^r$. 
  For any given sequence of signs ${\sf s} = (s_1,
  \ldots, s_r) \in \{+1,-1\}^r$, define 
  \[
  C_d^{s} = 
  C_d^{s_1, \ldots, s_r} = \{ \prod_{i=1}^r (\id + s_i \tau_i)\circ c : 
  c \in C_d\}.
  \]
  Similarly to the proof of Proposition~\ref{formula-prop}, we may
  write the identity as a sum
  \[
  \sum_{{\sf s} \in \{+1,-1\}^r}\frac{1}{2^r}\prod_{i=1}^r (\id + s_i
  \tau_i)
  \]
  of endomorphisms.
  Since the endomorphisms are mutually orthogonal and
  idempotent, we obtain the following generalization of
  Proposition~\ref{formula-prop}.
  \begin{proposition}
    With notation as above, we have that 
    \[
    \mathcal{C} = \bigoplus_{s \in \{+1,-1\}^r}
    \mathcal{C}^{s}.
    \]
    In particular, 
    \[
    H_d(\mathcal{C}) \cong 
    \bigoplus_{s \in \{+1,-1\}^r}
    H_d(\mathcal{C}^{s}). 
    \]
    \label{genformula-prop}
  \end{proposition}

  For the matching complex $\M[n]$, a natural choice of involutions
  is $\tau_1 = (1,2), \tau_2 = (3, 4), \ldots, \tau_r = (2r-1,2r)$,
  where $2r \le n$. Defining $\lambda = (2^r,1^{n-2r}) = (2, \ldots,
  2, 1, \ldots, 1)$, we obtain that
  \[
  \tilde{C}_d(\M[n];R)/(\lambda,{\sf s})
  = \{ \prod_{i=1}^r (\id + s_i\tau_i) \circ c : c \in C_d(\M[n];R)
  \}.
  \]
  By Proposition~\ref{genformula-prop} and
  Theorem~\ref{corechaincpx-thm}, we have that
  \[
  \tilde{H}_d(\M[n];R) \cong
  \bigoplus_{s \in \{+1,-1\}^r}
  \tilde{H}_d(\M[n];R)/(\lambda,{\sf s}) \cong 
  \bigoplus_{s \in \{+1,-1\}^r}
  \tilde{H}_d((\lambda,{\sf s});R).
  \]
  Define $(+)^a(-)^b$ to be the sequence consisting of $a$ $+$ signs 
  followed by $b$ $-$ signs.
  \begin{proposition}
    With notation as above, we have that
    \[
    \tilde{H}_d(\M[n];R) \cong
    \bigoplus_{i=0}^r \bigoplus_{\binom{r}{i}}
    \tilde{H}_d((\lambda,(+)^{r-i}(-)^i);R).
    \]
  \end{proposition}
  \begin{proof}
    This is immediate from the fact that 
    $\mathcal{C}((\lambda,{\sf s});R)$ and
    $\mathcal{C}((\lambda,{\sf s}');R)$ are isomorphic whenever 
    $s$ and $s'$ consist of the same number of $+$
    signs.
  \end{proof}

  We use this decomposition to analyze the homology of $\M[n]$ for $5
  \le n \le 16$; see Section~\ref{upto16-sec}.

  \section{Overview of computations}

  The purpose of this section is to present an overview of the
  computations leading to the results presented in 
  Section~\ref{intro-sec}. For more detailed information, we refer
  to the tables at the end of this paper.

  \subsection{The homology of $\M[n]$ for $n\le 16$}
  \label{upto16-sec}

  Applying the decomposition described in
  Section~\ref{abeliantwo-sec}, we analyze the homology of
  $\M[n]$ for $5 \le n \le 16$ using 
  the computer program {\sf HOMCHAIN} \cite{Pilar};
  see Tables~\ref{m5-tab}--\ref{m16-tab} for a summary. 
  Since the underlying group has order a power of $2$,
  our computations do not give us any immediate information about the 
  the existence of elements of order $2$ in the homology of $\M[n]$.
  For this reason, we express our results in terms of the coefficient
  ring
  \[
  \tZ = \{a\cdot 2^{-b} : a,b \in \Z, b \ge 0 \},
  \]
  thus ignoring the Sylow $2$-subgroup of the torsion part of the
  homology.

  In some cases, we only managed to compute the homology
  over finite fields $\Z_p$ and not over $\tZ$. Using the Universal
  Coefficient Theorem, one may still obtain the 
  $p$-rank of the homology; see
  Hatcher \cite[Cor. 3A.6]{Hatcher}.  We refer to
  Theorems~\ref{5tor-thm} and \ref{3tor-thm} and
  Table~\ref{matching-tab} for a summary of our results.

  The next value, $n=17$, is certainly not out of reach for the
  existing software, but it does not seem worth the considerable
  effort to compute the homology of $\M[17]$. We already know the
  homology in the top dimension $7$ \cite{Bouc} and also have a
  qualified guess about the homology in degree $6$ (see
  Conjecture~\ref{3torrank-conj}). This leaves us with the rank of
  the elementary $3$-group in degree $5$ (the smallest degree with
  nonvanishing homology), likely a random-looking number that will not
  tell us much useful unless we also get to know the rank of the
  bottom nonvanishing homology of $\M[n]$ for more values of $n$.

  Define $f_0 = 1$, $f_1 = 0$, and $f_i = f_{i-1}+f_{i-2}$ for $i \ge
  2$; these are the Fibonacci numbers. By Table~\ref{m14-tab}, the
  $5$-rank of
  $\tilde{H}_4(((2,2,2,2,2,2,2),(+)^{7-i}(-)^i);\tZ)$ is equal to
  $f_i$.  
  It is not hard to show that this implies the following; 
  see Section~\ref{lambdascc-sec} for the definition of 
  $\BD{r}{\lambda}$.
  \begin{corollary}
    For $0 \le i \le 7$, the $5$-rank of
    of $\tilde{H}_4(\BD{7+i}{2^{7-i} 1^{2i}};\tZ)$ is
    $f_{2i}$.
  \end{corollary}
  It is known \cite{Andersen,matchtor} that the Sylow $5$-subgroup is
  elementary for $i=0$.

  \subsection{Detecting elements of order $5$}

  To detect elements of order $5$ in $\tilde{H}_{4+u}(\M[14+2u];\Z)$ for $0 \le
  u \le 4$, we use the action induced by 
  \[
  \left(\begin{array}{cccccc} \qw{\lambda_1} & \qw{\lambda_2} &
      \qw{\lambda_3} & \qw{\lambda_4} & \qw{\lambda_5} &
      \qw{\lambda_6}
      \\ 
      \qw{-} & \qw{-} & \qw{-} & \qw{-} & \qw{-} & \qw{-}
    \end{array}\right)
  \]
  for choices of partitions $(\lambda_1, \lambda_2, \lambda_3,
  \lambda_4, \lambda_5, \lambda_6)$ such that  
  $2 \le \lambda_a \le 4$ for all $a$. 
  This does not bring us beyond $u = 4$, but a similar action based
  on a partition with seven instead of six parts helps us find
  elements of order $5$ for $5 \le u \le 7$. The resulting chain
  complexes are too large for our computer to handle, so we reduce
  them further using group actions of the form 
  \[
  \left\{
    4 \times 
    \begin{array}{c}
      \qw{\lambda_1}\\ 
      \qw{-}
    \end{array}\right\}
  \join
    \left(\begin{array}{ccc} \qw{\lambda_2} & \qw{\lambda_3} &
    \qw{\lambda_4}
        \\ 
        \qw{-} & \qw{-} & \qw{-}
      \end{array}\right).
  \]
  Yet another construction yields elements of order $5$ in
  $\tilde{H}_{12}(\M[30];\Z)$.
  See Tables~\ref{bdtwist5even-tab} and
  \ref{bdtwist5even2-tab} for a summary of our computations 
  for even $n$.

  The same kind of action turns out to yield
  elements of order $5$ in the group $\tilde{H}_{6+u}(\M[19+2u];\Z)$
  for $0 \le u \le 4$; see Table~\ref{bdtwist5odd-tab}.
  Moreover, the specific action represented as 
  \[
  \left\{
    4 \times 
    \begin{array}{c}
      \qw{3}\\ 
      \qw{-}
    \end{array}\right\}
  \join 
  \left\{
    4 \times 
    \begin{array}{c}
      \qw{3}\\ 
      \qw{-}
    \end{array}\right\}
  \]
  yields elements of order $5$ in $\tilde{H}_{8}(\M[24];\Z)$; see
  Table~\ref{bdtwist7even-tab}.

  \subsection{Detecting elements of order $7$}

  To detect elements of order $7$ in $\tilde{H}_{8+u}(\M[23+2u];\Z)$
  for $0 \le u \le 9$, we use partitions with $8$ parts 
  as summarized in Table~\ref{bdtwist7odd-tab}. 
  The actions are of the forms
  \[
  \left\{
    5 \times 
    \begin{array}{c}
      \qw{\lambda_1}\\ 
      \qw{-}
    \end{array}\right\}
  \join
    \left(\begin{array}{ccc} \qw{\lambda_2} & \qw{\lambda_3} &
    \qw{\lambda_4}
        \\ 
        \qw{-} & \qw{-} & \qw{-}
      \end{array}\right)
    \mbox{\ \ \ and\ \ \ }
  \left\{
    2 \times 
    \begin{array}{c}
      \qw{\lambda_1}\\ 
      \qw{-}
    \end{array}\right\}
  \join
  \left\{
    3 \times 
    \begin{array}{c}
      \qw{\lambda_2}\\ 
      \qw{-}
    \end{array}\right\}
  \join
  \left\{
    3 \times 
    \begin{array}{c}
      \qw{\lambda_3}\\ 
      \qw{-}
    \end{array}\right\}.
  \]
  We detect elements of order $7$ in $\tilde{H}_{11}(\M[30];\Z)$
  by analyzing the complex induced by the 
  action represented as 
  \[
  \left\{
    3 \times 
    \begin{array}{c}
      \qw{2}\\ 
      \qw{-}
    \end{array}\right\}
  \join 
  \left\{
    6 \times 
    \begin{array}{c}
      \qw{4}\\ 
      \qw{-}
    \end{array}\right\};
  \]
 see Table~\ref{bdtwist7even-tab}.

  \subsection{Detecting elements of order $11$ and $13$}

  As alluded to in the introduction, one may disclose
  elements of order $(2r-1)$ in the group
  $\tilde{H}_{\binom{r+1}{2}-2}(\M[(r+1)^2-2];\Z)$ 
  for $2r-1 \in \{5,7,11,13\}$
  by analyzing the complex induced by the action
  \[
  \left\{ (r+1) \times \begin{array}{c}
      \qw{r}\\ 
      \qw{-} 
    \end{array}\right\} 
  \join
  \left(
    \begin{array}{c}
      \qw{r-1}\\ 
      \qw{+}
    \end{array}\right);
  \]
  see Table~\ref{bdtwist4-tab}.
  We have not been able to detect any further elements of
  order $11$ or $13$, let alone larger primes $p$.

  \section{Acknowledgment}
  
  I thank an anonymous referee for many helpful comments and
  suggestions, for pointing out quite a few inconsistencies in an
  earlier manuscript, and for providing simplifications to the proofs
  of Propositions~\ref{formula-prop} and \ref{genformula-prop}.
  The idea of looking at direct products of wreath products emerged
  from fruitful discussions with Dave Benson.

  \begin{footnotesize}
  
  \end{footnotesize}

  \newpage

  \begin{table}[htb]
    \caption{The groups
      $\tilde{H}_d(((2,2,1),\mathsf{s});\tZ)$ yielding
        $\tilde{H}_d(\M[5];\tZ)$.}
  \begin{footnotesize}
    \begin{center}
      \begin{tabular}{|c||l||c|}
        \hline
        & &\\[-2.5ex]
        $\mathsf{s}$
        &  $d=1$ & factor
        \\
        & &\\[-2.5ex]
        \hline
        \hline
        & &\\[-2.5ex]
        $+ +$ & $\tZ$ & 1 \\
        & &\\[-2.5ex]
        \hline
        & &\\[-2.5ex]
        $+ -$  & $\tZ^{2}$ & 2 \\
        & &\\[-2.5ex]
        \hline
        & &\\[-2.5ex]
        $- -$  & $\tZ$ & 1 \\
        & &\\[-2.5ex]
        \hline
        \hline
        & &\\[-2.5ex]
        $\tilde{H}_d(\M[5];\tZ)$
        & $\tZ^{6}$ & \\[-2.5ex]
        & &\\
        \hline
      \end{tabular}
    \end{center}
  \end{footnotesize}
  \label{m5-tab}
  \end{table}

  \begin{table}[ht]
    \caption{The groups
      $\tilde{H}_d(((2,2,2),\mathsf{s});\tZ)$ yielding
        $\tilde{H}_d(\M[6];\tZ)$.}
  \begin{footnotesize}
    \begin{center}
      \begin{tabular}{|c||l||c|}
        \hline
        & &\\[-2.5ex]
        $\mathsf{s}$
        &  $d=1$ & factor
        \\
        & &\\[-2.5ex]
        \hline
        \hline
        & &\\[-2.5ex]
        $+++$ & $\tZ^{2}$ & 1 \\
        & &\\[-2.5ex]
        \hline
        & &\\[-2.5ex]
        $++-$  & $\tZ^{2}$ & 3 \\
        & &\\[-2.5ex]
        \hline
        & &\\[-2.5ex]
        $+--$  & $\tZ^{2}$ & 3 \\
        & &\\[-2.5ex]
        \hline
        & &\\[-2.5ex]
        $---$  & $\tZ^{2}$ & 1 \\
        & &\\[-2.5ex]
        \hline
        \hline
        & &\\[-2.5ex]
        $\tilde{H}_d(\M[6];\tZ)$
        & $\tZ^{16}$ & \\[-2.5ex]
        & &\\
        \hline
      \end{tabular}
    \end{center}
  \end{footnotesize}
    \label{m6-tab}
  \end{table}

  \begin{table}[ht]
    \caption{The groups
      $\tilde{H}_d(((2,2,2,1),\mathsf{s});\tZ)$ yielding
        $\tilde{H}_d(\M[7];\tZ)$.}
      \begin{footnotesize}
    \begin{center}
      \begin{tabular}{|c||l|l||c|}
        \hline
        & & &\\[-2.5ex]
        $\mathsf{s}$
        &  $d=1$ & $d=2$ & factor
        \\
        & & &\\[-2.5ex]
        \hline
        \hline
        & & &\\[-2.5ex]
        $+++$ & $-$ & $\tZ$ & 1 \\
        & & &\\[-2.5ex]
        \hline
        & & &\\[-2.5ex]
        $++-$  & $-$ & $\tZ^{3}$ & 3 \\
        & & &\\[-2.5ex]
        \hline
        & & &\\[-2.5ex]
        $+--$  & $-$ & $\tZ^{3}$ & 3 \\
        & & &\\[-2.5ex]
        \hline
        & & &\\[-2.5ex]
        $---$  & $\Z_3$ & $\tZ$ & 1 \\
        & & &\\[-2.5ex]
        \hline
        \hline
        & & &\\[-2.5ex]
        $\tilde{H}_d(\M[7];\tZ)$
        & $\Z_3$ & $\tZ^{20}$ & \\[-2.5ex]
        & & &\\
        \hline
      \end{tabular}
    \end{center}
    \end{footnotesize}
    \label{m7-tab}
  \end{table}

  \begin{table}[ht] 
    \caption{The groups
      $\tilde{H}_d(((2,2,2,2),\mathsf{s});\tZ)$ yielding
        $\tilde{H}_d(\M[8];\tZ)$.}
      \begin{footnotesize}
    \begin{center}
      \begin{tabular}{|c||l||c|}
        \hline
        & &\\[-2.5ex]
        $\mathsf{s}$
        &  $d=2$ & factor
        \\
        & &\\[-2.5ex]
        \hline
        \hline
        & &\\[-2.5ex]
        $++++$ & $\tZ^{6}$ & 1 \\
        & &\\[-2.5ex]
        \hline
        & &\\[-2.5ex]
        $+++-$  & $\tZ^{9}$ & 4 \\
        & &\\[-2.5ex]
        \hline
        & &\\[-2.5ex]
        $++--$  & $\tZ^{8}$ & 6 \\
        & &\\[-2.5ex]
        \hline
        & &\\[-2.5ex]
        $+---$  & $\tZ^{9}$ & 4 \\
        & &\\[-2.5ex]
        \hline
        & &\\[-2.5ex]
        $----$  & $\tZ^{6}$ & 1 \\
        & &\\[-2.5ex]
        \hline
        \hline
        & &\\[-2.5ex]
        $\tilde{H}_d(\M[8];\tZ)$
        & $\tZ^{132}$ & \\[-2.5ex]
        & &\\
        \hline
      \end{tabular}
    \end{center}
    \end{footnotesize}
    \label{m8-tab}
  \end{table}

  \begin{table}[ht] 
    \caption{The groups
      $\tilde{H}_d(((2,2,2,2,1),\mathsf{s});\tZ)$ yielding
        $\tilde{H}_d(\M[9];\tZ)$.}
      \begin{footnotesize}
    \begin{center}
      \begin{tabular}{|c||l|l||c|}
        \hline
        & & &\\[-2.5ex]
        $\mathsf{s}$
        &  $d=2$ & $d=3$ & factor
        \\
        & & &\\[-2.5ex]
        \hline
        \hline
        & & &\\[-2.5ex]
        $++++$ & $\tZ^3$ & $\tZ$ & 1 \\
        & & &\\[-2.5ex]
        \hline
        & & &\\[-2.5ex]
        $+++-$  & $\tZ^3$ & $\tZ^{4}$ & 4 \\
        & & &\\[-2.5ex]
        \hline
        & & &\\[-2.5ex]
        $++--$  & $\tZ^2$ & $\tZ^{6}$ & 6 \\
        & & &\\[-2.5ex]
        \hline
        & & &\\[-2.5ex]
        $+---$  & $\tZ^3 \oplus \Z_3$ & $\tZ^{4}$ & 4 \\
        & & &\\[-2.5ex]
        \hline
        & & &\\[-2.5ex]
        $----$  & $\tZ^3 \oplus (\Z_3)^4$ & $\tZ$ & 1 \\
        & & &\\[-2.5ex]
        \hline
        \hline
        & & &\\[-2.5ex]
        $\tilde{H}_d(\M[9];\tZ)$
        & $\tZ^{42} + (\Z_3)^8$ & $\tZ^{70}$ & \\[-2.5ex]
        & & &\\
        \hline
      \end{tabular}
    \end{center}
    \end{footnotesize}
    \label{m9-tab}
  \end{table}

  \begin{table}[ht] 
    \caption{The groups
      $\tilde{H}_d(((2,2,2,2,2),\mathsf{s});\tZ)$ yielding
        $\tilde{H}_d(\M[10];\tZ)$.}
      \begin{footnotesize}
    \begin{center}
      \begin{tabular}{|c||l|l||c|}
        \hline
        & & &\\[-2.5ex]
        $\mathsf{s}$
        &  $d=2$ & $d=3$ & factor
        \\
        & & &\\[-2.5ex]
        \hline
        \hline
        & & &\\[-2.5ex]
        $+++++$ & $-$ & $\tZ^{28}$ & 1 \\
        & & &\\[-2.5ex]
        \hline
        & & &\\[-2.5ex]
        $++++-$  & $-$ & $\tZ^{36}$ & 5 \\
        & & &\\[-2.5ex]
        \hline
        & & &\\[-2.5ex]
        $+++--$  & $-$ & $\tZ^{40}$ & 10 \\
        & & &\\[-2.5ex]
        \hline
        & & &\\[-2.5ex]
        $++---$  & $-$ & $\tZ^{40}$ & 10 \\
        & & &\\[-2.5ex]
        \hline
        & & &\\[-2.5ex]
        $+----$  & $-$ & $\tZ^{36}$ & 5 \\
        & & &\\[-2.5ex]
        \hline
        & & &\\[-2.5ex]
        $-----$  & $\Z_3$ & $\tZ^{28}$ & 1 \\
        & & &\\[-2.5ex]
        \hline
        \hline
        & & &\\[-2.5ex]
        $\tilde{H}_d(\M[10];\tZ)$
        & $\Z_3$ & $\tZ^{1216}$ & \\[-2.5ex]
        & & &\\
        \hline
      \end{tabular}
    \end{center}
    \end{footnotesize}
    \label{m10-tab}
  \end{table}

  \begin{table}[ht] 
    \caption{The groups
      $\tilde{H}_d(((2,2,2,2,2,1),\mathsf{s});\tZ)$ yielding
        $\tilde{H}_d(\M[11];\tZ)$.}
      \begin{footnotesize}
    \begin{center}
      \begin{tabular}{|c||l|l||c|}
        \hline
        & & &\\[-2.5ex]
        $\mathsf{s}$
        &  $d=3$ & $d=4$ & factor
        \\
        & & &\\[-2.5ex]
        \hline
        \hline
        & & &\\[-2.5ex]
        $+++++$ & $\tZ^{39}$ & $\tZ$ & 1 \\
        & & &\\[-2.5ex]
        \hline
        & & &\\[-2.5ex]
        $++++-$ & $\tZ^{39}$ & $\tZ^{5}$ & 5 \\
        & & &\\[-2.5ex]
        \hline
        & & &\\[-2.5ex]
        $+++--$ & $\tZ^{36}$ & $\tZ^{10}$ & 10 \\
        & & &\\[-2.5ex]
        \hline
        & & &\\[-2.5ex]
        $++---$ & $\tZ^{36}+\Z_3$ & $\tZ^{10}$ & 10 \\
        & & &\\[-2.5ex]
        \hline
        & & &\\[-2.5ex]
        $+----$ & $\tZ^{39}+(\Z_3)^5$ & $\tZ^{5}$ & 5 \\
        & & &\\[-2.5ex]
        \hline
        & & &\\[-2.5ex]
        $-----$ & $\tZ^{39}+(\Z_3)^{10}$ & $\tZ$ & 1 \\
        & & &\\[-2.5ex]
        \hline
        \hline
        & & &\\[-2.5ex]
        $\tilde{H}_d(\M[11];\tZ)$         
        & $\tZ^{1188}+(\Z_3)^{45}$ & $\tZ^{252}$ & \\[-2.5ex]
        & & &\\
        \hline
      \end{tabular}
    \end{center}
    \end{footnotesize}
    \label{m11-tab}
  \end{table}

  \begin{table}[ht] 
    \caption{The groups
      $\tilde{H}_d(((2,2,2,2,2,2),\mathsf{s});\tZ)$ yielding
        $\tilde{H}_d(\M[12];\tZ)$.}
      \begin{footnotesize}
    \begin{center}
      \begin{tabular}{|c||l|l||c|}
        \hline
        & & &\\[-2.5ex]
        $\mathsf{s}$
        &  $d=3$ & $d=4$ & factor
        \\
        & & &\\[-2.5ex]
        \hline
        \hline
        & & &\\[-2.5ex]
        $++++++$ & $-$ & $\tZ^{140}$ & 1 \\
        & & &\\[-2.5ex]
        \hline
        & & &\\[-2.5ex]
        $+++++-$ & $-$ & $\tZ^{170}$ & 6 \\
        & & &\\[-2.5ex]
        \hline
        & & &\\[-2.5ex]
        $++++--$ & $-$ & $\tZ^{200}$ & 15 \\
        & & &\\[-2.5ex]
        \hline
        & & &\\[-2.5ex]
        $+++---$ & $-$ & $\tZ^{206}$ & 20 \\
        & & &\\[-2.5ex]
        \hline
        & & &\\[-2.5ex]
        $++----$ & $\Z_3$ & $\tZ^{200}$ & 15 \\
        & & &\\[-2.5ex]
        \hline
        & & &\\[-2.5ex]
        $+-----$ & $(\Z_3)^{5}$ & $\tZ^{170}$ & 6 \\
        & & &\\[-2.5ex]
        \hline
        & & &\\[-2.5ex]
        $------$ & $(\Z_3)^{11}$ & $\tZ^{140}$ & 1 \\
        & & &\\[-2.5ex]
        \hline
        \hline
        & & &\\[-2.5ex]
        $\tilde{H}_d(\M[12];\tZ)$         
        & $(\Z_3)^{56}$ & $\tZ^{12440}$ & \\[-2.5ex]
        & & & \\
        \hline
      \end{tabular}
    \end{center}
    \end{footnotesize}
    \label{m12-tab}
  \end{table}

  \begin{table}[ht] 
    \caption{The groups
      $\tilde{H}_d(((2,2,2,2,2,2,1),\mathsf{s});\tZ)$ yielding
        $\tilde{H}_d(\M[13];\tZ)$. In boxes marked with
        $(*)$, the 
        torsion part is a guess based on computations over $\Z_3$ and
        some additional small fields $\Z_p$.}
      \begin{footnotesize}
    \begin{center}
      \begin{tabular}{|c||l|l|l||c|}
        \hline
        & & & & \\[-2.5ex]
        $\mathsf{s}$
        &  $d=3$ & $d=4$ & $d=5$ & factor
        \\
        & & & & \\[-2.5ex]
        \hline
        \hline
        & & & & \\[-2.5ex]
        $++++++$ & $-$ & $\tZ^{369}$ & $\tZ$ & 1 \\
        & & & & \\[-2.5ex]
        \hline
        & & & & \\[-2.5ex]
        $+++++-$ & $-$ & $\tZ^{384}$ & $\tZ^{6}$ & 6 \\
        & & & & \\[-2.5ex]
        \hline
        & & & & \\[-2.5ex]
        $++++--$ & $-$ & $\tZ^{387}$ & $\tZ^{15}$ & 15 \\
        & & & & \\[-2.5ex]
        \hline
        & & & & \\[-2.5ex]
        $+++---$ & $-$ & $\tZ^{382}+\Z_3$ & $\tZ^{20}$ & 20 \\
        & & & & \\[-2.5ex]
        \hline
        & & & & \\[-2.5ex]
        $++----$ & $-$ & $\tZ^{387}+(\Z_3)^6$ & $\tZ^{15}$ & 15 \\
        & & & & \\[-2.5ex]
        \hline
        & & & & \\[-2.5ex]
        $+-----$ & $-$ & $\tZ^{384}+(\Z_3)^{15}$ \hfill$\mbox{}^{(*)}$& $\tZ^{6}$ & 6 \\
        & & & & \\[-2.5ex]
        \hline
        & & & & \\[-2.5ex]
        $------$ & $\Z_3$ & $\tZ^{369}+(\Z_3)^{20}$ & $\tZ$ & 1 \\
        & & & & \\[-2.5ex]
        \hline
        \hline
        & & & & \\[-2.5ex]
        $\tilde{H}_d(\M[13];\tZ)$ & $\Z_3$
        & $\tZ^{24596} + (\Z_3)^{220}$ \hfill$\mbox{}^{(*)}$& $\tZ^{924}$ & \\[-2.5ex]
        & & & & \\
        \hline
      \end{tabular}
    \end{center}
    \end{footnotesize}
    \label{m13-tab}
  \end{table}

  \begin{table}[ht] 
    \caption{The groups
      $\tilde{H}_d(((2,2,2,2,2,2,2),\mathsf{s});\tZ)$ yielding
      $\tilde{H}_d(\M[14];\tZ)$. The torsion parts are
      guesses 
      based on computations over $\Z_3$, $\Z_5$, and $\Z_7$.}
    \begin{footnotesize}
    \begin{center}
      \begin{tabular}{|c||l|l||c|}
        \hline
        & & & \\[-2.5ex]
        $\mathsf{s}$
        &  $d=4$ & $d=5$ & factor
        \\
        & & & \\[-2.5ex]
        \hline
        \hline
        & & & \\[-2.5ex]
        $+++++++$ & $\Z_5$ & $\tZ^{732}$ & 1 \\
        & & & \\[-2.5ex]
        \hline
        & & & \\[-2.5ex]
        $++++++-$ & $-$ & $\tZ^{900}$ & 7 \\
        & & & \\[-2.5ex]
        \hline
        & & & \\[-2.5ex]
        $+++++--$ & $\Z_5$ & $\tZ^{1052}$ & 21 \\
        & & & \\[-2.5ex]
        \hline
        & & & \\[-2.5ex]
        $++++---$ & $\Z_5+(\Z_3)^{3}$ & $\tZ^{1140}$ & 35 \\
        & & & \\[-2.5ex]
        \hline
        & & & \\[-2.5ex]
        $+++----$ & $(\Z_5)^{2}+(\Z_3)^{15}$ & $\tZ^{1140}$ & 35 \\
        & & & \\[-2.5ex]
        \hline
        & & & \\[-2.5ex]
        $++-----$ & $(\Z_5)^{3}+(\Z_3)^{40}$ & $\tZ^{1052}$ & 21 \\
        & & & \\[-2.5ex]
        \hline
        & & & \\[-2.5ex]
        $+------$ & $(\Z_5)^{5}+(\Z_3)^{81}$ & $\tZ^{900}$ & 7 \\
        & & & \\[-2.5ex]
        \hline
        & & & \\[-2.5ex]
        $-------$ & $(\Z_5)^{8}+(\Z_3)^{120}$ & $\tZ^{732}$ & 1 \\
        & & & \\[-2.5ex]
        \hline
        \hline
        & & & \\[-2.5ex]
        $\tilde{H}_d(\M[14];\tZ)$         
        & $(\Z_5)^{233}+(\Z_3)^{2157}$ & $\tZ^{138048}$ & \\[-2.5ex]
        & & & \\
        \hline
      \end{tabular}
    \end{center}
    \end{footnotesize}
    \label{m14-tab}
  \end{table}

  \begin{table}[ht] 
    \caption{The groups
      $\tilde{H}_d(((2,2,2,2,2,2,2,1),\mathsf{s});\tZ)$
      yielding $\tilde{H}_d(\M[15];\tZ)$.
      The torsion parts for $d=5$ are guesses
      based on computations over $\Z_3$ and $\Z_5$.}
    \begin{footnotesize}
    \begin{center}
      \begin{tabular}{|c||l|l|l||c|}
        \hline
        & & & & \\[-2.5ex]
        $\mathsf{s}$
        &  $d=4$ & $d=5$ & $d=6$ & factor
        \\
        & & & & \\[-2.5ex]
        \hline
        \hline
        & & & & \\[-2.5ex]
        $+++++++$ & $-$ & $\tZ^{3309}$ &$\tZ$& 1 \\
        & & & & \\[-2.5ex]
        \hline
        & & & & \\[-2.5ex]
        $++++++-$ & $-$ & $\tZ^{3543}$ &$\tZ^{7}$& 7 \\
        & & & & \\[-2.5ex]
        \hline
        & & & & \\[-2.5ex]
        $+++++--$ & $-$ & $\tZ^{3689}$ &$\tZ^{21}$& 21 \\
        & & & & \\[-2.5ex]
        \hline
        & & & & \\[-2.5ex]
        $++++---$ & $-$ & $\tZ^{3739} + \Z_3$ &$\tZ^{35}$& 35 \\
        & & & & \\[-2.5ex]
        \hline
        & & & & \\[-2.5ex]
        $+++----$ & $-$ & $\tZ^{3739}+ (\Z_3)^{7}$ &$\tZ^{35}$& 35 \\
        & & & & \\[-2.5ex]
        \hline
        & & & & \\[-2.5ex]
        $++-----$ & $\Z_3$ & $\tZ^{3689}+ (\Z_3)^{21}$ &$\tZ^{21}$& 21 \\
        & & & & \\[-2.5ex]
        \hline
        & & & & \\[-2.5ex]
        $+------$ & $(\Z_3)^{7}$ & $\tZ^{3543}+ (\Z_3)^{35}$ &$\tZ^{7}$& 7 \\
        & & & & \\[-2.5ex]
        \hline
        & & & & \\[-2.5ex]
        $-------$ & $(\Z_3)^{22}$ & $\tZ^{3309}+ (\Z_3)^{35}$ &$\tZ$& 1 \\
        & & & & \\[-2.5ex]
        \hline
        \hline
        & & & & \\[-2.5ex]
        $\tilde{H}_d(\M[15];\tZ)$         
        & $(\Z_3)^{92}$ & $\tZ^{472888} + (\Z_3)^{1001}$ &
        $\tZ^{3432}$ & 
        \\[-2.5ex]
        & & & & \\
        \hline
      \end{tabular}
    \end{center}
    \end{footnotesize}
    \label{m15-tab}
  \end{table}

  \begin{table}[ht] 
    \caption{The groups
        $\tilde{H}_d(((2,2,2,2,2,2,2,2),\mathsf{s});\tZ)$
        yielding $\tilde{H}_d(\M[16];\tZ)$.
        The torsion parts for
        $d=5$ are guesses based on computations over $\Z_3$, $\Z_5$,
        and  $\Z_7$.}
      \begin{footnotesize}
    \begin{center}
      \begin{tabular}{|c||l|l|l||c|}
        \hline
        & & & & \\[-2.5ex]
        $\mathsf{s}$
        &  $d=4$ & $d=5$ & $d=6$ & factor
        \\
        & & & & \\[-2.5ex]
        \hline
        \hline
        & & & & \\[-2.5ex]
        {$++++++++$} & $-$ & $\tZ^{126}+(\Z_5)^{7}$ &$\tZ^{4060}$& 1 \\
        & & & & \\[-2.5ex]
        \hline
        & & & & \\[-2.5ex]
        {$+++++++-$} & $-$ & $\tZ^{105} + (\Z_5)^{7}$ &$\tZ^{5019}$& 8 \\
        & & & & \\[-2.5ex]
        \hline
        & & & & \\[-2.5ex]
        {$++++++--$} & $-$ & $\tZ^{100} + (\Z_5)^{12} + (\Z_3)^{10}$ &$\tZ^{5894}$& 28 \\
        & & & & \\[-2.5ex]
        \hline
        & & & & \\[-2.5ex]
        {$+++++---$} & $-$ & $\tZ^{91}\ + (\Z_5)^{18} + (\Z_3)^{55}$ &$\tZ^{6545}$& 56 \\
        & & & & \\[-2.5ex]
        \hline
        & & & & \\[-2.5ex]
        {$++++----$} & $-$ & $\tZ^{90}\ + (\Z_5)^{27}+ (\Z_3)^{159}$ &$\tZ^{6768}$& 70 \\
        & & & & \\[-2.5ex]
        \hline
        & & & & \\[-2.5ex]
        {$+++-----$} & $-$ & $\tZ^{91}\ + (\Z_5)^{41} + (\Z_3)^{350}$ &$\tZ^{6545}$& 56 \\
        & & & & \\[-2.5ex]
        \hline
        & & & & \\[-2.5ex]
        {$++------$} & $-$ & $\tZ^{100} + (\Z_5)^{61} + (\Z_3)^{635}$ &$\tZ^{5894}$& 28 \\
        & & & & \\[-2.5ex]
        \hline
        & & & & \\[-2.5ex]
        {$+-------$} & $-$ & $\tZ^{105} + (\Z_5)^{91} + (\Z_3)^{966}$ &$\tZ^{5019}$& 8 \\
        & & & & \\[-2.5ex]
        \hline
        & & & & \\[-2.5ex]
        {$--------$} & $\Z_3$ & $\tZ^{126} + (\Z_5)^{134} + (\Z_3)^{1253}$ &$\tZ^{4060}$& 1 \\
        & & & & \\[-2.5ex]
        \hline
        \hline
        & & & & \\[-2.5ex]
        $\tilde{H}_d(\M[16];\tZ)$         
        & $\Z_3$ & $\tZ^{24024} + (\Z_5)^{8163}$ & 
        $\tZ^{1625288}$ & \\
        &        & $+ (\Z_3)^{60851}$ & 
         & \\[-2.5ex]
        & & & & \\
        \hline
      \end{tabular}
    \end{center}
    \end{footnotesize}
    \label{m16-tab}
  \end{table}

  \begin{table}[ht]
   \caption{Examples yielding elements of order $5$ in
     $\tilde{H}_{4+u}(\M[14+2u];\Z)$ for $0 \le u \le 4$; 
      $\rho_d(p) = \dim \Tor(\tilde{H}_d((G,\alpha);\Z);\Z_p)$ 
      is the $p$-rank of $\tilde{H}_d((G,\alpha);\Z)$
      and $\beta_d$ is the free rank of the same group.}
    \begin{footnotesize}
      \begin{center}
      \begin{tabular}{|c|c|c||c||c|c|c|c|}
        \hline
        & & & & & & & \\[-2.3ex]
        Complex & Group action $(G,\alpha)$ & $d$ 
        & 
        $\beta_d$
        &
        $\rho_d(2)$ &
        $\rho_d(3)$ & $\rho_d(5)$ & 
        $\rho_d(7)$ \\
        & & & & & & & \\[-2.3ex]
        \hline
        \hline 
        & & & & & & & \\[-2.3ex]
        $\M[14]$&
        \multirow{2}{2.3cm}{%
        $\left(\begin{array}{cccccc}
              \qw{2}&\qw{2}&\qw{2}&\qw{2}&\qw{3}&\qw{3}\\ 
              \qw{-}&\qw{-}&\qw{-}&\qw{-}&\qw{-}&\qw{-}
            \end{array}\right)$ 
        }
          & $4$ & $-$   & $-$ &  $19$  &   $2$   &   $-$   \\
        & & $5$ & $112$ & $-$ &  $-$   &   $-$   &   $-$   \\
        & & & & & & & \\[-2.3ex]
        \hline
        \hline 
        & & & & & & & \\[-2.3ex]
        $\M[16]$&
        \multirow{2}{2.3cm}{%
        $\left(\begin{array}{cccccc}
              \qw{2}&\qw{2}&\qw{3}&\qw{3}&\qw{3}&\qw{3}\\ 
              \qw{-}&\qw{-}&\qw{-}&\qw{-}&\qw{-}&\qw{-}
            \end{array}\right)$ 
        }
        &   $5$ & $6$  & $-$ &  $43$  &   $8$   &   $-$   \\
        & & $6$ & $76$ & $-$ &  $-$   &   $-$   &   $-$   \\
        & & & & & & & \\[-2.3ex]
        \hline
        \hline 
        & & & & & & & \\[-2.3ex]
        $\M[18]$&
        \multirow{2}{2.3cm}{%
        $\left(\begin{array}{cccccc}
              \qw{2}&\qw{2}&\qw{3}&\qw{3}&\qw{4}&\qw{4}\\ 
              \qw{-}&\qw{-}&\qw{-}&\qw{-}&\qw{-}&\qw{-}
            \end{array}\right)$ 
        }
        &   $6$ & $20$ & $-$ &  $18$  &   $2$   &   $-$   \\
        & & $7$ & $20$ & $-$ &  $-$   &   $-$   &   $-$   \\
        & & & & & & & \\[-2.3ex]
        \cline{2-8}
        & & & & & & & \\[-2.3ex]
        &
        \multirow{2}{2.3cm}{%
        $\left(\begin{array}{cccccc}
              \qw{2}&\qw{3}&\qw{3}&\qw{3}&\qw{3}&\qw{4}\\ 
              \qw{-}&\qw{-}&\qw{-}&\qw{-}&\qw{-}&\qw{-}
            \end{array}\right)$ 
        }
        &   $6$ & $28$ & $-$ &  $44$  &   $8$   &   $-$   \\
        & & $7$ & $28$ & $-$ &  $- $  &   $-$   &   $-$   \\
        & & & & & & & \\[-2.3ex]
        \cline{2-8} 
        & & & & & & & \\[-2.3ex]
        &
        \multirow{2}{2.3cm}{%
        $\left(\begin{array}{cccccc}
              \qw{3}&\qw{3}&\qw{3}&\qw{3}&\qw{3}&\qw{3}\\ 
              \qw{-}&\qw{-}&\qw{-}&\qw{-}&\qw{-}&\qw{-}
            \end{array}\right)$ 
        }
        &   $6$ & $40$ & $-$ &  $90$  &   $20$  &   $-$   \\
        & & $7$ & $40$ & $-$ &  $-$   &   $-$   &   $-$   \\
        & & & & & & & \\[-2.3ex]
        \hline
        \hline 
        & & & & & & & \\[-2.3ex]
        $\M[20]$&
        \multirow{2}{2.3cm}{%
        $\left(\begin{array}{cccccc}
              \qw{3}&\qw{3}&\qw{3}&\qw{3}&\qw{4}&\qw{4}\\ 
              \qw{-}&\qw{-}&\qw{-}&\qw{-}&\qw{-}&\qw{-}
            \end{array}\right)$ 
        }
        &   $7$ & $76$ & $-$ &  $43$  &   $8$   &   $-$   \\
        & & $8$ & $6 $ & $-$ &  $-$   &   $-$   &   $-$   \\
        & & & & & & & \\[-2.3ex]
        \hline
        \hline 
        & & & & & & & \\[-2.3ex]
        $\M[22]$&
        \multirow{2}{2.3cm}{%
        $\left(\begin{array}{cccccc}
              \qw{3}&\qw{3}&\qw{4}&\qw{4}&\qw{4}&\qw{4}\\ 
              \qw{-}&\qw{-}&\qw{-}&\qw{-}&\qw{-}&\qw{-}
            \end{array}\right)$ 
        }
        &   $8$ & $112$ & $-$ &  $19$  &   $2$   &   $-$   \\
        & & $9$ & $-  $ & $-$ &  $-$   &   $-$   &   $-$ 
        \\[-2.3ex]
        & & & & & & & \\
        \hline
      \end{tabular}
    \end{center}
    \end{footnotesize}
    \label{bdtwist5even-tab}
  \end{table}

  \begin{table}[ht]
   \caption{Examples yielding elements of order $5$ in
     $\tilde{H}_{4+u}(\M[14+2u];\Z)$ for $4 \le u \le 8$
     (notation as in Table~\ref{bdtwist5even-tab}).}
    \begin{footnotesize}
      \begin{center}
      \begin{tabular}{|c|c|c||c||c|c|c|c|}
        \hline
        & & & & & & & \\[-2.3ex]
        Complex & Group action $(G,\alpha)$ & $d$ 
        & 
        $\beta_d$
        &
        $\rho_d(2)$ &
        $\rho_d(3)$ & $\rho_d(5)$ & 
        $\rho_d(7)$ \\
        & & & & & & & \\[-2.3ex]
        \hline
        \hline 
        & & & & & & & \\[-2.3ex]
        $\M[22]$&
        \multirow{2}{2.9cm}{%
          $
          \left\{
              4 \times 
            \begin{array}{c}
              \qw{3}\\ 
              \qw{-}
            \end{array}\right\}
          \join \left(
            \begin{array}{ccc}
              \qw{3}&\qw{3}&\qw{4}\\ 
              \qw{-}&\qw{-}&\qw{-} 
            \end{array}\right)
          $
        } & $8$ & $58$  & $36$ &  $53$  &  $14$  &   $-$   \\
        & & $9$ & $-$   & $-$  &  $-$   &  $-$   &   $-$   \\
        & & & & & & & \\[-2.3ex]
        \hline
        \hline 
        & & & & & & & \\[-2.3ex]
        $\M[24]$&
        \multirow{2}{2.9cm}{%
          $
          \left\{
              4 \times 
            \begin{array}{c}
              \qw{3}\\ 
              \qw{-}
            \end{array}\right\}
          \join \left(
            \begin{array}{ccc}
              \qw{4}&\qw{4}&\qw{4}\\ 
              \qw{-}&\qw{-}&\qw{-} 
            \end{array}\right)
          $
        } & $9$  & $93$ & $32$ &  $29$  &   $5$   &   $-$   \\
        & & $10$ & $-$  & $1$  &  $-$   &   $-$   &   $-$   \\
        & & & & & & & \\[-2.3ex]
        \hline
        \hline 
        & & & & & & & \\[-2.3ex]
        $\M[26]$&
        \multirow{2}{2.9cm}{%
          $
          \left\{
              4 \times 
            \begin{array}{c}
              \qw{4}\\ 
              \qw{-}
            \end{array}\right\}
          \join \left(
            \begin{array}{ccc}
              \qw{3}&\qw{3}&\qw{4}\\ 
              \qw{-}&\qw{-}&\qw{-} 
            \end{array}\right)
          $
        } & $10$ & $141$ & $25$ &  $22$  &   $2$   &   $-$   \\
        & & $11$ & $-$   & $-$  &  $-$   &   $-$   &   $-$   \\
        & & & & & & & \\[-2.3ex]
        \hline
        \hline 
        & & & & & & & \\[-2.3ex]
        $\M[28]$&
        \multirow{2}{2.9cm}{%
          $
          \left\{
              4 \times 
            \begin{array}{c}
              \qw{4}\\ 
              \qw{-}
            \end{array}\right\} 
          \join \left(
            \begin{array}{ccc}
              \qw{4}&\qw{4}&\qw{4}\\ 
              \qw{-}&\qw{-}&\qw{-} 
            \end{array}\right)
          $
        } & $11$ & $167$ & $35$ &  $18$  &   $3$   &   $-$   \\
        & & $12$ & $-$   & $-$  &  $-$   &   $-$   &   $-$   \\
        & & & & & & & \\[-2.3ex]
        \hline
        \hline 
        & & & & & & & \\[-2.3ex]
        $\M[30]$&
        \multirow{2}{3.6cm}{%
          $
          \left\{
              4 \times 
            \begin{array}{c}
              \qw{4}\\ 
              \qw{-}
            \end{array}\right\} 
          \join 
          \left\{
              3 \times 
            \begin{array}{c}
              \qw{4}\\
              \qw{-}
            \end{array}\right\} 
          \join
          \left(
            \begin{array}{c}
              \qw{2}\\ 
              \qw{-}
            \end{array}\right)
          $
        } & $11$ & $-$   & $-$  &  $8$   &   $-$   &   $-$   \\
        & & $12$ & $550$ & $53$ &  $29$  &   $2$   &   $-$   
        \\[-2.3ex]
        & & & & & & & \\
        \hline
      \end{tabular}
    \end{center}
    \end{footnotesize}
    \label{bdtwist5even2-tab}
  \end{table}

  \begin{table}[ht]
   \caption{Examples yielding elements of order $5$ in
     $\tilde{H}_{6+u}(\M[19+2u];\Z)$ for $0 \le u \le 4$
     (notation as in Table~\ref{bdtwist5even-tab}).}
    \begin{footnotesize}
     \begin{center}
      \begin{tabular}{|c|c|c||c||c|c|c|c|}
        \hline
        & & & & & & & \\[-2.3ex]
        Complex & Group action $(G,\alpha)$ & $d$  
        & 
        $\beta_d$
        &
        $\rho_d(2)$ & 
        $\rho_d(3)$ & $\rho_d(5)$ & 
        $\rho_d(7)$ \\
        & & & & & & & \\[-2.3ex]
        \hline
        \hline 
        & & & & & & & \\[-2.3ex]
        $\M[19]$&
        \multirow{2}{2.9cm}{%
          $
          \left\{
              4 \times 
            \begin{array}{c}
              \qw{3}\\ 
              \qw{-}
            \end{array}\right\}
          \join \left(
            \begin{array}{ccc}
              \qw{2}&\qw{2}&\qw{3}\\ 
              \qw{-}&\qw{-}&\qw{-} 
            \end{array}\right)
          $
        } & $6$ & $-$  & $-$  &  $4$   &   $1$   &   $-$   \\
        & & $7$ & $96$ & $5$  &  $1$   &   $-$   &   $-$   \\
        & & & & & & & \\[-2.3ex]
        \hline
        \hline 
        & & & & & & & \\[-2.3ex]
        $\M[21]$&
        \multirow{2}{2.9cm}{%
          $
          \left\{
              4 \times 
            \begin{array}{c}
              \qw{3}\\ 
              \qw{-}
            \end{array}\right\}
          \join \left(
            \begin{array}{ccc}
              \qw{3}&\qw{3}&\qw{3}\\ 
              \qw{-}&\qw{-}&\qw{-} 
            \end{array}\right)
          $
        } & $7$ & $-$   & $1$  &  $20$   &   $3$   &   $-$   \\
        & & $8$ & $167$ & $7$  &  $-$   &   $-$   &   $-$   \\
        & & & & & & & \\[-2.3ex]
        \hline
        \hline 
        & & & & & & & \\[-2.3ex]
        $\M[23]$&
        \multirow{2}{2.9cm}{%
          $
          \left\{
              4 \times 
            \begin{array}{c}
              \qw{3}\\ 
              \qw{-}
            \end{array}\right\}
          \join \left(
            \begin{array}{ccc}
              \qw{3}&\qw{4}&\qw{4}\\ 
              \qw{-}&\qw{-}&\qw{-} 
            \end{array}\right)
          $
        } & $8$ & $-$   & $-$  &  $22$   &   $2$   &   $-$   \\
        & & $9$ & $141$ & $-$  &  $2$   &   $-$   &   $-$   \\
        & & & & & & & \\[-2.3ex]
        \hline
        \hline 
        & & & & & & & \\[-2.3ex]
        $\M[25]$&
        \multirow{2}{2.9cm}{%
          $
          \left\{
              4 \times 
            \begin{array}{c}
              \qw{4}\\ 
              \qw{-}
            \end{array}\right\}
          \join \left(
            \begin{array}{ccc}
              \qw{3}&\qw{3}&\qw{3}\\ 
              \qw{-}&\qw{-}&\qw{-} 
            \end{array}\right)
          $
        } & $9$  & $-$  & $2$  &  $27$   &   $5$   &   $-$   \\
        & & $10$ & $93$ & $1$  &  $1$   &   $-$   &   $-$   \\
        & & & & & & & \\[-2.3ex]
        \hline
        \hline 
        & & & & & & & \\[-2.3ex]
        $\M[27]$&
        \multirow{2}{2.9cm}{%
          $
          \left\{
              4 \times 
            \begin{array}{c}
              \qw{4}\\ 
              \qw{-}
            \end{array}\right\}
          \join \left(
            \begin{array}{ccc}
              \qw{3}&\qw{4}&\qw{4}\\ 
              \qw{-}&\qw{-}&\qw{-} 
            \end{array}\right)
          $
        } & $10$ & $-$  & $12$  &  $51$   &   $14$   &   $-$   \\
        & & $11$ & $58$ & $2$  &  $2$   &   $-$   &   $-$   
        \\[-2.3ex]
        & & & & & & & \\
        \hline
      \end{tabular}
    \end{center}
    \end{footnotesize}
    \label{bdtwist5odd-tab}
  \end{table}

  \begin{table}[ht]
   \caption{Examples yielding elements of order $7$ in
     $\tilde{H}_{8+u}(\M[23+2u];\Z)$ for $0 \le u \le 9$
     (notation as in Table~\ref{bdtwist5even-tab}).}
    \begin{footnotesize}
     \begin{center}
      \begin{tabular}{|c|c|c||c||c|c|c|}
        \hline
        & & & & & & \\[-2.3ex]
        Complex & Group action $(G,\alpha)$ & $d$  
        & 
        $\beta_d$
        &
        $\rho_d(3)$ & $\rho_d(5)$ & 
        $\rho_d(7)$ 
        \\
        & & & & & & \\[-2.3ex]
        \hline
        \hline 
        & & & & & & \\[-2.3ex]
        $\M[23]$
        &
          \multirow{2}{4.5cm}{%
          $
          \left\{
              3 \times 
            \begin{array}{c}
              \qw{2}\\ 
              \qw{-}
            \end{array}\right\} 
          \join 
          \left\{
              3 \times 
            \begin{array}{c}
              \qw{3}\\ 
              \qw{-}
            \end{array}\right\} 
          \join 
          \left\{
              2 \times 
            \begin{array}{c}
              \qw{4}\\ 
              \qw{-}
            \end{array}\right\}
          $
        } & $8$ & $-$  &  $29$   &   $3$   &   $1$   \\
        & & $9$ & $332$ &  $-$   &   $-$   &   $-$   
        \\
        & & & & & & \\[-2.3ex]
        \hline
        \hline 
        & & & & & & \\[-2.3ex]
        $\M[25]$
        &
          \multirow{2}{4.5cm}{%
          $
          \left\{
              2 \times 
            \begin{array}{c}
              \qw{2}\\ 
              \qw{-}
            \end{array}\right\} 
          \join 
          \left\{
              3 \times 
            \begin{array}{c}
              \qw{3}\\ 
              \qw{-}
            \end{array}\right\} 
          \join 
          \left\{
              3 \times 
            \begin{array}{c}
              \qw{4}\\ 
              \qw{-}
            \end{array}\right\}
          $
        } & $9$ & $-$  &  $73$   &   $11$   &   $2$   \\
        & & $10$ & $407$ &  $1$   &   $-$   &   $-$   
        \\
        & & & & & & \\[-2.3ex]
        \hline
        \hline 
        & & & & & & \\[-2.3ex]
        $\M[27]$
        &
          \multirow{2}{2.9cm}{%
          $
          \left\{
              5 \times 
            \begin{array}{c}
              \qw{4}\\ 
              \qw{-}
            \end{array}\right\} 
          \join \left(
            \begin{array}{ccc}
              \qw{2}&\qw{2}&\qw{3}\\ 
              \qw{-}&\qw{-}&\qw{-} 
            \end{array}\right)
          $
        } & $10$ & $-$  &  $77$   &   $27$   &   $9$   \\
        &                 & $11$ & $150$ &  $1$   &   $-$   &   $-$   
        \\
        & & & & & & \\[-2.3ex]
        \hline
        \hline 
        & & & & & & \\[-2.3ex]
        $\M[29]$&
        \multirow{2}{2.9cm}{%
          $
          \left\{
              5 \times 
            \begin{array}{c}
              \qw{4}\\ 
              \qw{-}
            \end{array}\right\} 
          \join \left(
            \begin{array}{ccc}
              \qw{2}&\qw{2}&\qw{5}\\ 
              \qw{-}&\qw{-}&\qw{-} 
            \end{array}\right)
          $
        } & $11$ & $4$  &  $84$   &   $33$   &   $7$   \\
        &                 & $12$ & $66$ &  $-$   &   $-$   &   $-$   
        \\
        & & & & & & \\[-2.3ex]
        \hline
        \hline 
        & & & & & & \\[-2.3ex]
        $\M[31]$&
        \multirow{2}{2.9cm}{%
          $
          \left\{
              5 \times 
            \begin{array}{c}
              \qw{4}\\ 
              \qw{-}
            \end{array}\right\} 
          \join \left(
            \begin{array}{ccc}
              \qw{2}&\qw{3}&\qw{6}\\ 
              \qw{-}&\qw{-}&\qw{-} 
            \end{array}\right)
          $
        } & $12$ & $9$  &  $98$   &   $36$   &   $4$   \\
        &                 & $13$ & $32$ &  $1$   &   $-$   &   $-$   
        \\
        & & & & & & \\[-2.3ex]
        \hline
        \hline 
        & & & & & & \\[-2.3ex]
        $\M[33]$&
        \multirow{2}{2.9cm}{%
          $
          \left\{
              5 \times 
            \begin{array}{c}
              \qw{4}\\ 
              \qw{-}
            \end{array}\right\} 
          \join \left(
            \begin{array}{ccc}
              \qw{2}&\qw{5}&\qw{6}\\ 
              \qw{-}&\qw{-}&\qw{-} 
            \end{array}\right)
          $
        } & $13$ & $32$  &  $98$   &   $36$   &   $4$   \\
        &                 & $14$ & $9$ &  $1$   &   $-$   &   $-$   
        \\
        & & & & & & \\[-2.3ex]
        \hline
        \hline 
        & & & & & & \\[-2.3ex]
        $\M[35]$&
        \multirow{2}{2.9cm}{%
          $
          \left\{
              5 \times 
            \begin{array}{c}
              \qw{4}\\ 
              \qw{-}
            \end{array}\right\} 
          \join \left(
            \begin{array}{ccc}
              \qw{3}&\qw{6}&\qw{6}\\ 
              \qw{-}&\qw{-}&\qw{-} 
            \end{array}\right)
          $
        } & $14$ & $66$  &  $94$   &   $33$   &   $7$   \\
        &                 & $15$ & $4$ &  $-$   &   $-$   &   $-$   
        \\
        & & & & & & \\[-2.3ex]
        \hline
        \hline 
        & & & & & & \\[-2.3ex]
        $\M[37]$&
        \multirow{2}{2.9cm}{%
          $
          \left\{
              5 \times 
            \begin{array}{c}
              \qw{4}\\ 
              \qw{-}
            \end{array}\right\} 
          \join \left(
            \begin{array}{ccc}
              \qw{5}&\qw{6}&\qw{6}\\ 
              \qw{-}&\qw{-}&\qw{-} 
            \end{array}\right)
          $
        } & $15$ & $150$  &  $86$   &   $27$   &   $9$   \\
        &                 & $16$ & $-$ &  $-$   &   $-$   &   $-$   
        \\
        & & & & & & \\[-2.3ex]
        \hline
        \hline 
        & & & & & & \\[-2.3ex]
        $\M[39]$
        &
          \multirow{2}{4.5cm}{%
          $
          \left\{
              3 \times 
            \begin{array}{c}
              \qw{4}\\ 
              \qw{-}
            \end{array}\right\} 
          \join 
          \left\{
              3 \times 
            \begin{array}{c}
              \qw{5}\\ 
              \qw{-}
            \end{array}\right\} 
          \join 
          \left\{
              2 \times 
            \begin{array}{c}
              \qw{6}\\ 
              \qw{-}
            \end{array}\right\}
          $
        } & $16$ & $407$  &  $75$   &   $11$   &   $2$   \\
        & & $17$ & $-$ &  $1$   &   $-$   &   $-$   
        \\
        & & & & & & \\[-2.3ex]
        \hline
        \hline 
        & & & & & & \\[-2.3ex]
        $\M[41]$
        &
          \multirow{2}{4.5cm}{%
          $
          \left\{
              2 \times 
            \begin{array}{c}
              \qw{4}\\ 
              \qw{-}
            \end{array}\right\} 
          \join 
          \left\{
              3 \times 
            \begin{array}{c}
              \qw{5}\\ 
              \qw{-}
            \end{array}\right\} 
          \join 
          \left\{
              3 \times 
            \begin{array}{c}
              \qw{6}\\ 
              \qw{-}
            \end{array}\right\}
          $
        } & $17$ & $332$  &  $30$   &   $3$   &   $1$   \\
        & & $18$ & $-$ &  $-$   &   $-$   &   $-$   
        \\[-2.3ex]
        & & & & & & \\
        \hline
      \end{tabular}
    \end{center}
    \end{footnotesize}
    \label{bdtwist7odd-tab}
  \end{table}

  \begin{table}[ht]
   \caption{One example yielding elements of order $5$ in
     $\tilde{H}_{8}(\M[24];\Z)$ and another yielding elements of order
     $7$ in $\tilde{H}_{11}(\M[30];\Z)$
     (notation as in Table~\ref{bdtwist5even-tab}).}
    \begin{footnotesize}
     \begin{center}
      \begin{tabular}{|c|c|c||c||c|c|c|}
        \hline
        & & & & & & \\[-2.3ex]
        Complex & Group action $(G,\alpha)$ & $d$  
        & 
        $\beta_d$
        &
        $\rho_d(3)$ & $\rho_d(5)$ & 
        $\rho_d(7)$ 
        \\
        & & & & & & \\[-2.3ex]
        \hline
        \hline 
        & & & & & & \\[-2.3ex]
        $\M[24]$ &
        \multirow{2}{2.9cm}{%
          $
          \left\{
              4 \times 
            \begin{array}{c}
              \qw{3}\\ 
              \qw{-}
            \end{array}\right\}
          \join 
          \left\{
              4 \times 
            \begin{array}{c}
              \qw{3}\\ 
              \qw{-}
            \end{array}\right\}
          $
        } 
        & $8$ & $-$  &  $-$   &   $1$   &   $-$   \\
        & & $9$ & $67$  &  $28$   &   $6$   &   $-$   \\
        &                 & $10$ & $-$ &  $2$   &   $-$   &   $-$   \\
        & & & & & & \\[-2.3ex]
        \hline
        \hline 
        & & & & & & \\[-2.3ex]
        $\M[30]$
        &
          \multirow{2}{3.0cm}{%
          $
          \left\{
              3 \times 
            \begin{array}{c}
              \qw{2}\\ 
              \qw{-}
            \end{array}\right\} 
          \join 
          \left\{
              6 \times 
            \begin{array}{c}
              \qw{4}\\ 
              \qw{-}
            \end{array}\right\} 
          $
        } & $11$ & $-$  &  $6$   &   $-$   &   $1$   \\
        & & $12$ & $174$&  $8$   &   $-$   &   $-$    \\
        & & $13$ & $1$  &   $2$   &   $-$   &   $-$   
        \\[-2.3ex]
        & & & & & & \\
        \hline
      \end{tabular}
    \end{center}
    \end{footnotesize}
    \label{bdtwist7even-tab}
  \end{table}

  \begin{table}[ht]
   \caption{The integral homology 
     $\tilde{H}_d((G,\alpha);\Z)$ for certain choices 
     of parameters.}
    \begin{footnotesize}
     \begin{center}
      \begin{tabular}{|c|l|c||c||c|c|c|}
        \hline
        & & & \\[-2.3ex]
        Complex & Group action $(G,\alpha)$ & $d$  
        & 
        $\tilde{H}_d((G,\alpha);\Z)$ \\
        & & & \\[-2.3ex]
        \hline
        \hline 
        & & & \\[-2.3ex]
        $\M[7]$&
        \multirow{2}{2.9cm}{%
          $
          \left(
            \begin{array}{c}
              \qw{1}\\ 
              \qw{+}
            \end{array}\right)
          \join
          \left\{ 3 \times \begin{array}{c}
                \qw{2}\\ 
              \qw{-} 
            \end{array}\right\} 
          $
        } & $1$ & $\Z_3$  \\
        &                 & $2$ & $-$   \\
        & & & \\[-2.3ex]
        \hline
        \hline 
        & & & \\[-2.3ex]
        $\M[14]$&
        \multirow{2}{2.9cm}{%
          $
          \left(
            \begin{array}{c}
              \qw{2}\\ 
              \qw{+}
            \end{array}\right)
          \join
          \left\{ 4 \times \begin{array}{c}
                \qw{3}\\ 
              \qw{-} 
            \end{array}\right\}
          $
        } & $4$ & $\Z_5$  \\
        &                 & $5$ & $-$  \\
        & & & \\[-2.3ex]
        \hline
        \hline 
        & & & \\[-2.3ex]
        $\M[23]$&
        \multirow{2}{2.9cm}{%
          $
          \left(
            \begin{array}{c}
              \qw{3}\\ 
              \qw{+}
            \end{array}\right)
          \join
          \left\{ 5 \times \begin{array}{c}
                \qw{4}\\ 
              \qw{-} 
            \end{array}\right\}
          $
        } & $8$ & $\Z_7$   \\
        &                 & $9$ & $\Z$ \\
        & & & \\[-2.3ex]
        \hline
        \hline 
        & & & \\[-2.3ex]
        $\M[34]$&
        \multirow{2}{2.9cm}{%
          $
          \left(
            \begin{array}{c}
              \qw{4}\\ 
              \qw{+}
            \end{array}\right)
          \join
          \left\{ 6 \times \begin{array}{c}
                \qw{5}\\ 
              \qw{-} 
            \end{array}\right\}
          $
        } & $13$ & $\Z_9$   \\
        & &  $14$ & $\Z^6 \oplus \Z_3$   \\
        & & & \\[-2.3ex]
        \hline
        \hline 
        & & & \\[-2.3ex]
        $\M[47]$&
        \multirow{2}{2.9cm}{%
          $
          \left(
            \begin{array}{c}
              \qw{5}\\ 
              \qw{+}
            \end{array}\right)
          \join
          \left\{ 7 \times \begin{array}{c}
                \qw{6}\\ 
              \qw{-} 
            \end{array}\right\}
          $
        } & $19$ &  $\Z_{11}$   \\
        &                 & $20$ & $\Z^{18} \oplus (\Z_3)^2$   \\
        &                 & $21$ & $\Z_3 \oplus \Z_2$   \\
        & & & \\[-2.3ex]
        \hline
        \hline 
        & & & \\[-2.3ex]
        $\M[62]$&
        \multirow{2}{2.9cm}{%
          $
          \left(
            \begin{array}{c}
              \qw{6}\\ 
              \qw{+}
            \end{array}\right)
          \join
          \left\{ 8 \times \begin{array}{c}
                \qw{7}\\ 
              \qw{-} 
            \end{array}\right\}
          $
        } & $26$ & $\Z_{13}$  \\
        &                 & $27$ & $\Z^{55} \oplus (\Z_3)^6$ \\
        &                 & $28$ & $(\Z_3)^{3} \oplus (\Z_2)^4$ \\[-2.3ex]
        & & & \\
        \hline
      \end{tabular}
    \end{center}
    \end{footnotesize}
    \label{bdtwist4-tab}
  \end{table}

\end{document}